\documentclass[11pt, a4paper]{amsart}

\usepackage[T1]{fontenc}
\usepackage{microtype}

\usepackage{amsmath}								%math
\usepackage{amssymb}
\usepackage{amsthm}
\usepackage{amscd}
\usepackage{amsfonts}
\usepackage{newpxtext}
\usepackage{stmaryrd}

\usepackage{euler}
\usepackage{extarrows}
\usepackage[bookmarks=true, colorlinks=true, linkcolor=blue!50!black,
citecolor=orange!50!black, urlcolor=orange!50!black, pdfencoding=unicode, backref=page]{hyperref}
\usepackage{color}

\usepackage{tikz}									
\usetikzlibrary{matrix}
\usetikzlibrary{patterns}
\usetikzlibrary{matrix}
\usetikzlibrary{positioning}
\usetikzlibrary{decorations.pathmorphing}
\usetikzlibrary{cd}

\usepackage[all,cmtip]{xy}

\usepackage{fullpage}
\usepackage{enumerate}

\newtheorem*{theorem*}{Theorem}

\newtheorem{maintheorem}{Theorem}[section]
\newtheorem{theorem}{Theorem}[section]
\newtheorem{lemma}[theorem]{Lemma}
\newtheorem{proposition}[theorem]{Proposition}
\newtheorem{corollary}[theorem]{Corollary} 

\theoremstyle{definition}
\newtheorem{definition}[theorem]{Definition}
\newtheorem{example}[theorem]{Example}

\newtheorem{remark}[theorem]{Remark}

\newtheoremstyle{myitemstyle}						% flexible theorem style
	{}			%Space above
	{}			%Space below
	{}			%Body font
	{}			%indent amount
	{}			%Thm head font
	{.}			%Punkte nach thm head
	{ }			%Abstand nach thm head
	{}			%Thm head spec
\theoremstyle{myitemstyle}
\newtheorem{myitemthm}{}

\newcommand{\hooklongrightarrow}{\lhook\joinrel\longrightarrow}

\newcommand{\R}{\mathbb{R}}
\newcommand{\Rbar}{\overline{\mathbb{R}}}
\newcommand{\Z}{\mathbb{Z}}
\newcommand{\ZZ}{\mathbb{Z}}
\newcommand{\Q}{\mathbb{Q}}
\renewcommand{\C}{\mathbb{C}}
\newcommand{\CC}{\mathbb{C}}
\newcommand{\N}{\mathbb{N}}
\newcommand{\PP}{\mathbb{P}}
\newcommand{\NN}{\mathbb{N}}
\newcommand{\QQ}{\mathbb{Q}}

\newcommand{\A}{\mathbb{A}}
\renewcommand{\G}{\mathbb{G}}

\newcommand{\calB}{\mathcal{B}}

\newcommand{\frakf}{\mathfrak{f}}

\newcommand{\calO}{\mathcal{O}}

\newcommand{\calS}{\mathcal{S}}

\newcommand{\calV}{\mathcal{V}}

\DeclareMathOperator{\Spec}{Spec}
\DeclareMathOperator{\Hom}{Hom}

\DeclareMathOperator{\val}{val}

\DeclareMathOperator{\Trop}{Trop}

\DeclareMathOperator{\id}{id}

\DeclareMathOperator{\Span}{Span}

\DeclareMathOperator{\trop}{trop}

\DeclareMathOperator{\Proj}{Proj}
\DeclareMathOperator{\Rel}{Rel}
\DeclareMathOperator{\rk}{rk}
\DeclareMathOperator{\Quot}{Quot}
\DeclareMathOperator{\supp}{supp}
\DeclareMathOperator{\ev}{ev}

\newcommand{\deq}{\ensuremath{\stackrel{\textrm{def}}{=}}}
\newcommand{\rtl}{\ensuremath{\dashrightarrow}}
\newcommand{\equ}{\ensuremath{\,=\,}}
\newcommand{\dsubseteq}{\ensuremath{\,\subseteq\,}}

\title{Tropicalization of toric prevarieties} 

\date{}

\author{Alex K\"uronya}
\address{Institut f\"ur Mathematik, Goethe--Universit\"at Frankfurt,
60325 Frankfurt am Main, Germany}
\email{kuronya@math.uni-frankfurt.de}

\author{Pedro Souza}
\address{Institut f\"ur Mathematik, Goethe--Universit\"at Frankfurt,
60325 Frankfurt am Main, Germany}
\email{souza@math.uni-frankfurt.de}

\author{Martin Ulirsch}
\address{Institut f\"ur Mathematik, Goethe--Universit\"at Frankfurt,
60325 Frankfurt am Main, Germany}
\email{ulirsch@math.uni-frankfurt.de}

\begin{document}

\maketitle

\begin{abstract}
The homogeneous spectrum of a multigraded finitely generated algebra (in the sense of Brenner-Schr\"oer) always admits an embedding into a toric variety that is not necessarily separated, a so-called \emph{toric prevariety}. In order to have a convenient framework to study the tropicalization of homogeneous spectra we propose a tropicalization procedure for toric prevarieties and study its basic properties. With these tools at hand, we prove a generalization of Payne's and Foster--Gross--Payne's tropical limit theorem for divisorial schemes.
\end{abstract}

\setcounter{tocdepth}{1}
\tableofcontents

%%%%%%%%%%%%%%%%%%%%%%%%%%%%%%%%%%%%%%%%%%%%%%%%%%%%%%

\section*{Introduction}

 Non-separated compactifications of moduli spaces appear very naturally in algebraic geometry; among other examples compactified Jacobians over semistable degenerations of algebraic curves come to mind. The upshot is that (possibly) non-separated compactifications often  have a chance to be a lot closer to the original moduli problem than compactifications that are forced to be separated. 

Tropical geometry, as initially introduced by Mikhalkin (see e.g.  \cite{Mikhalkin_ICM}), has become an important tool to study the combinatorial geometry of compactifications of moduli spaces (see in particular \cite{ACP, CCUW}). For example, in the case of compactified Jacobians the combinatorics of the different possible limits of line bundles in the special fiber can be recovered in terms of the chip-firing combinatorics of divisors on the dual graph of a semistable curve (see \cite{BakerJensen} and the references within as well as, in particular, \cite{AbreuPacini, MMUV} for a perspective on this topic in terms of universal Jacobians). 

The first step when studying the tropicalization of an algebraic variety $Y$ over a non-Archimedean field $K$ is to choose some form of local or global coordinates. The most classical way to do this is by finding a closed embedding $i$ of $Y$ into a toric variety $X$ with big torus $T$. Then this embedding allows us to take coordinate-wise valuations in order to get a polyhedral object, the \emph{tropicalization} 
\begin{equation*}
\Trop_X(Y)\deq\Trop(Y,i)
\end{equation*}
of $Y$ with respect to the embedding $i\colon Y\hookrightarrow X$. We refer the reader to \cite{MaclaganSturmfels, Gubler_guide} for more technical details in the case of algebraic tori and to \cite{Kajiwara_troptoric, Payne_anallimittrop} in the case of toric varieties. We, in particular, point the reader to \cite{Payne_anallimittrop} and \cite{FGP}, where the authors use the tropicalization with respect to embeddings into toric varieties to study the structure of Berkovich analytic spaces.

There are, however, many cases, in which an embedding into a toric variety is not-well understood. One aspect of this problem is that non-separated varieties do not admit closed embeddings into a toric variety and so the phenomenon of non-separatedness cannot be captured by the process of tropicalization described above. We point to \cite{BrennerSchroer, KU} for the theory of multihomogeneous spectra of multigraded rings that form a very natural and rich source of non-separated algebraic varieties.

In this article we intend to offer a remedy to this particular problem by developing a theory of tropicalization for (possibly non-separated) \emph{toric prevarieties} (as introduced and studied in \cite{KKMSD_toroidal} and \cite{ACampoNeuenHausen_toricpre}). 

The central idea is that every toric prevariety $X$ (with big torus $T$) admits an open cover by $T$-invariant open affine subsets $U_\sigma=\Spec K[S_\sigma]$ (for a rational polyhedral cone $\sigma$) on which the process of tropicalization is well-understood by \cite{Kajiwara_troptoric, Payne_anallimittrop}. The main task of this paper is to glue these local pictures in a way that respects the combinatorial structure of $X$ that is described in terms of so-called \emph{systems of fans} in \cite{ACampoNeuenHausen_toricpre}.

It turns out that, contrary to the situation of (separated) toric varieties, there are really two different ways of doing this: The first approach goes by using the natural stratification of the affine tropical toric variety $U_\sigma^{trop}=\Hom(S_\sigma,\Rbar)$ at infinity and gluing $X^{trop}$ the same way the strata of $X$ are glued algebraically. The other approach is to consider the non-negative part $U_\sigma^{trop,\geq 0}=\Hom(S_\sigma,\Rbar_{\geq 0})$ and by gluing along the faces of $\sigma$ (which naturally endows $X^{trop,\geq 0}$ with the structure of an extended cone complex in the sense of \cite{ACP, Ulirsch_functroplogsch}). We refer to $X^{trop}$ as a \emph{tropical toric prevariety} and to $X^{trop,\geq 0}$ as a \emph{non-negative tropical toric prevariety}. 

When $X$ is a (separated) toric variety, the non-negative tropicalization $X^{trop,\geq 0}$ of $X$ is a subspace of $X^{trop}$, namely the closure of the fan of $X$ in $X^{trop}$. Contrary to the usual tropicalization of a toric variety, the  non-negative tropicalization can be generalized to toroidal embeddings (see \cite{Thuillier_toroidal, Ulirsch_functroplogsch}) and therefore plays an important role when studying the tropical geometry of moduli spaces.

We shall see in Section \ref{section_tropicalization} below that there is also a natural continuous, proper and surjective \emph{tropicalization map} 
\begin{equation*}
\trop_{X}\colon X^{an}\longrightarrow X^{trop}
\end{equation*}
from the Berkovich analytic space $X^{an}$ associated to $X$ to $X^{trop}$. 

Any toric prevariety $X$ is the base change of a certain monoid scheme defined $\mathbb{Z}$; so we may also speak of toric prevarieties  over the valuation ring $R$ of $K$ base change (or any ring for that matter). This is not strictly speaking a "variety" in the classical sense, but we prefer this terminology to avoid clumsy phrases (as, for example, the term "toric scheme" is already taken, see \cite{KKMSD_toroidal, Gubler_guide}).

We will see in Section \ref{section_positivetropicalization} below that we have a natural continuous and surjective \emph{non-negative tropicalization map} 
\begin{equation*}
\trop_{X}^{\geq 0}\colon X^{\circ}\longrightarrow X^{trop,\geq 0}
\end{equation*}
from the Raynaud generic fiber $X^{\circ}$ of a toric prevariety $X$ over the valution ring $R$ to $X^{trop,\geq 0}$. The relationship between the two objects can be summarized by the commutative diagram 
\begin{equation*}\begin{tikzcd}
X^{\circ}\arrow[rr,"\trop"]\arrow[d]&&X^{trop_X^{\geq 0},\geq 0}\arrow[d]\\
X_K^{an}\arrow[rr,"\trop_X"]&& X^{trop} \ .
\end{tikzcd}\end{equation*}
Here the vertical arrows are the natural morphisms $X^{\circ}\rightarrow X_K^{an}$ (sending an $R'$-valued point, for a valuation ring $R'$ extending $R$, to its generic fiber) and $X^{trop,\geq 0}\rightarrow X^{trop}$ (constructed in Proposition \ref{prop_nonnegtroptotrop} below). Both vertical arrows are injective, as soon as $X$ is separated, and bijective, when $X$ is proper.

Given a possibly separated scheme $Y$ of finite type over $K$ as well as a (locally closed) embedding $i\colon Y\hookrightarrow X$ into a toric prevariety $X$ (over $K$), we may now define the \emph{tropicalization} of $Y$ with respect to the embedding $i$ by
\begin{equation*}
    \Trop(Y,i)\deq\trop_X\big(i(Y)^{an}\big) \ .
\end{equation*}
Similarly, given a possibly separated scheme $Y$ of finite type over $R$ as well as a (locally closed) embedding $i\colon Y\hookrightarrow X$ into a toric prevariety $X$ (over $R$), we define the \emph{positive tropicalization} of $Y$ with respect to the embedding $i$ by
\begin{equation*}
    \Trop^{\geq 0}(Y,i)\deq\trop_X^{\geq 0}\big(i(Y)^{\circ}\big)\ .
\end{equation*}

The restriction of this process to every torus orbit is the usual tropicalization in the sense of \cite{MaclaganSturmfels, Gubler_guide} and so, by the Bieri--Groves Theorem \cite{BieriGroves, KapranovEinsiedlerLind}, the intersection of $\Trop(Y,i)$ (or $\Trop_{\geq 0}(Y,i)$ respectively) with every stratum of $X^{trop}$ (or $X^{trop,\geq 0}$ respectively) carries a (possibly non-unique) structure of a polyhedral complex that is rational with respect to the value group of $K$.

The process of tropicalization constructed here is functorial with respect to toric morphisms of the surrounding toric prevariety and so it makes sense to consider systems of embeddings into toric prevarieties. In fact, expanding on \cite{Payne_anallimittrop, FGP}, we have the following:
\begin{maintheorem}\label{maintheorem}
Let $K$ be a complete non-Archimedean field and $R$ its valuation ring.
\begin{enumerate}[(i)] 
\item Let $Y$ be a divisorial scheme of finite type over  $K$. Then the tropicalization maps associated to all embeddings $i\colon Y\rightarrow X$ into a simplicial toric prevariety $X$ naturally induce a homeomorphism
\begin{equation*}
Y^{an}\xlongrightarrow{\sim}\varprojlim \Trop(Y,i)
\end{equation*}
between the Berkovich analytic space $Y^{an}$ and the projective limit over all tropicalizations $\Trop(Y,i)$. 
\item Let $Y$ be a divisorial scheme of finite type over  $R$. Then the non-negative tropicalization maps associated to all embeddings $i\colon Y\rightarrow X$ into a simplicial toric prevariety $X$ naturally induce a homeomorphism
\begin{equation*}
Y^\circ\xlongrightarrow{\sim}\varprojlim \Trop^{\geq 0}(Y,i)
\end{equation*}
between the Raynaud generic fiber $Y^\circ$ and the projective limit over all $\Trop^{\geq 0}(Y,i)$.
\end{enumerate}
\end{maintheorem}

Our proof of Theorem \ref{maintheorem} mostly follows along the lines laid out in \cite{FGP} with one main exception. In order to construct a sufficient number of embeddings into toric varieties the authors of \cite{FGP} rely on some rather technical results coming from \cite{Wlodarczyk_toricembeddings}. We provide a more elementary alternative by working with the theory of multihomogeneous spectra of multigraded algebras, as developed in \cite{BrennerSchroer} and further expanded upon in \cite{KU}. The price to pay for this simplification is that, in general, we cannot work with closed embeddings into toric varieties, but only with locally closed embeddings into toric prevarieties, for which we develop the tropical machinery in this article.

\subsection*{Acknowledgements} The authors have been funded by the  Deutsche Forschungsgemeinschaft  (DFG, German Research Foundation) TRR 326 \emph{Geometry and Arithmetic of Uniformized Structures}, project number 444845124. A.~K. and M.~U. were  partially supported by the LOEWE grant ‘Uniformized Structures in Algebra and Geometry’. P.~S. acknowledges funding from the Deutscher Akademischer Austauschdienst (DAAD) under the scholarship program Forschungsstipendien - Promotionen in Deutschland, 2020/21 - No. 57507871.  M.~U. has also been funded by the Deutsche Forschungsgemeinschaft (DFG, German Research Foundation) - project number 456557832.

%%%%%%%%%%%%%%%%%%%%%%%%%%%%%%%%%%%%%%%%%%%%%%%%%%%%%%

\section{Toric prevarieties and systems of fans}

Let us first recall the description of the category of toric prevarieties in terms of systems of fans that was developed  by  A’Campo-Neuen and  Hausen in \cite{ACampoNeuenHausen_toricpre}. Here we begin by working over the complex numbers $\mathbb{C}$, toric prevarieties over arbitrary base rings are introduced in Definition \ref{def_otherbases} below.

\begin{definition}
A \emph{toric prevariety} is a normal integral scheme that is of finite type (but not necessarily separated) over $\CC$ together with an operation of an algebraic torus $T$, and the choice of a  point $x_0\in X$ such that $T\hookrightarrow X$ given by $t\mapsto t\cdot x_0$ is an open embedding. 
\end{definition}

A toric prevariety is a (normal) toric variety in the sense of \cite{Fulton_toric} or \cite{CoxLittleSchenk_toric} if and only it is separated. A morphism $f\colon X\rightarrow X'$ between toric prevarieties $(X,x_0)$ and $(X',x_0')$ is said to be a \emph{toric morphism}, if $f(x_0)=x_0'$ and there is a homomorphism $\phi\colon T\rightarrow T'$ of algebraic tori such that $f(t\cdot x)=\phi(t)\cdot f(x)$ for all $t\in T$ and $x\in X$. 

The category of (separated) toric varieties with toric varieties is equivalent to the category of (rational polyhedral) fans. In order to give a similar description for the category of toric prevarieties one needs to generalize this notion (see \cite[Section 2]{ACampoNeuenHausen_toricpre}). 

\begin{definition}
Let $N$ be a finitely generated free abelian group and write $N_\R$ for the vector space $N\otimes \R$. A collection $S=(\Delta_{ij})_{i,j\in I}$ of fans in $N_\R$ (indexed by a finite set $I$) is said to be a \emph{system of fans}, if 
\begin{enumerate}[(i)]
\item $\Delta_{ij}=\Delta_{ji}$ for all $i,j\in I$ and 
\item $\Delta_{ij}\cap\Delta_{jk}$ is a subfan of $\Delta_{ik}$ for all $i,j,k\in I$. 
\end{enumerate} 
\end{definition}

It immediately follows from these axioms that for all $i,j\in I$ the fan $\Delta_{ij}$ is a subfan of $\Delta_{ii}$. To a system of fans $S(\Delta_{ij})_{i,j\in I}$ we may naturally associate a toric prevariety $X=X(S)$ as follows: Consider the toric varieties $X_i=X(\Delta_{ii})$ for $i\in I$ and the open subvarieties $X_{ij}=X(\Delta_{ij})$ of $X_i$ respectively $X_j$ for $i,j\in I$ with $i\neq j$. By Axiom (i) we have a natural isomorphism $f_{ji}\colon X_{ij}\xrightarrow{\sim} X_{ji}$ and by Axiom (ii) we have $f_{ki}=f_{kj}\circ f_{ji}$ for all $i,j,k\in I$. Therefore we may glue the $X_i$ along the isomorphisms $X_{ij}\simeq X_{ji}$. This gluing is compatible with the respective operations of the torus $T:=N\otimes \C^\ast$ and therefore we obtain a toric prevariety $X(S)$. 

Every toric prevariety has two distinguished systems of charts: One is by maximal torus-invariant separated open subsets, the other by maximal torus-invariant affine open subsets. The latter corresponds to a system of fans $\calS$, in which every $\Delta_{ii}$ is a fan associated to a rational polyhedral cone. We refer to such systems of fans as \emph{affine} systems of fans.

By Sumihiro's Theorem every toric prevariety may be covered by finite many torus-invariant open affine subsets (see \cite[Prop. 1.3]{ACampoNeuenHausen_toricpre}). Therefore every toric prevariety is of the form $X(S)$ for a system of fans $S$. 

\begin{example}\label{example_A1twoorigins}
Let $I=\{1,2\}$ and consider the one-dimensional cone $\sigma=\mathbb{R}_{\geq 0}$. We define a system of fans $\calS$ in $\mathbb{R}$ by setting $\Delta_{11}\deq\Delta_{22}\deq\{\sigma,\{0\}\}$ and $\Delta_{12}\deq\Delta_{21}\deq\{\{0\}\}$. Then the corresponding toric prevariety $X(\calS)$ is given by taking two copies of $\A^1$, namely $X(\Delta_{11})$ and $X(\Delta_{22})$, and gluing them over $\G_m\subseteq \A^1$ via the identity $\G_m=X(\Delta_{12})=X(\Delta_{21})=\G_m$. So $X(\calS)$ is precisely the affine line with two origins.  
\end{example}

Let $S$ be a system of fans. Consider the set of pairs $(\sigma, i)$ consisting of a rational polyhedral cone $\sigma$ in $N_\R$ and an index $i\in I$ that fulfil $\sigma \in \Delta_{ii}$. Two such pairs $(\sigma, i)$ and $(\tau,j)$ are said to be \emph{equivalent} if $\sigma=\tau$ as cones and $\sigma=\tau\in\Delta_{ij}$. We write $\Omega(S)$ for the set of equivalence classes $[\sigma,i]$. There is a natural partial order on $\Omega(S)$ given by $[\tau,j]\preceq [\sigma,i]$ if and only if $\tau$ is a face of $\sigma$ and $[\tau,i]=[\tau,j]$.

\begin{definition}\label{def_morsysfan}
Let $S=(\Delta_{ij})_{i,j\in I}$ and $S'=(\Delta_{ij})_{i,j\in I'}$ be two system of fans in the vector space $N_\R$ and $N'_\R$ generated by a free finitely generated abelian groups $N$ and $N'$ respectively. A tuple consisting of a homomorphism $F\colon N\rightarrow N'$ and a map $\frakf\colon \Omega(S)\rightarrow \Omega(S')$ is said to be a \emph{morphism of systems of fans}, if 
\begin{enumerate}[(i)]
\item the map $\frakf$ is order preserving, i.e. whenever $[\tau,j]\preceq [\sigma,i]$ it follows that $\frakf([\tau,j])\preceq \frakf([\sigma,i])$, and 
\item whenever $f([\sigma,i])=[\sigma',i']$ then $F_\R(\sigma)\subseteq \sigma'$. 
\end{enumerate}
\end{definition}

For a system of fans $\calS$ we write $U_{[\sigma,i]}=U_\sigma=\Spec\mathbb{C}(S_\sigma)$ for the affine torus-invariant open subset associated to $[\sigma,i]\in\Omega(\calS)$. By Axiom (i) in Definition \ref{def_morsysfan} a morphism $(F,\frakf)\colon \calS\rightarrow \calS'$ of fans induces toric morphisms $U_{[\sigma,i]}\rightarrow U_{\frakf[\sigma,i]}$ that glue to a toric morphism $X(\calS)\rightarrow X(\calS')$ of toric prevarieties by Axiom (ii). 
This defines a functor from the category of systems of fans to the category of toric prevarieties with toric morphisms. In  \cite{ACampoNeuenHausen_toricpre} the authors have shown:

\begin{theorem}[\cite{ACampoNeuenHausen_toricpre} Theorem 3.6] The category of toric prevarieties with toric morphisms is equivalent to the category of systems of fans and to the full subcategory of affine systems of fans. 
\end{theorem}

Let $X=X(\calS)$ be a toric prevariety. By \cite[Theorem 4.1]{ACampoNeuenHausen_toricpre} there is a (separated) toric variety $\widetilde{X}$ together with a toric morphism $X\rightarrow \widetilde{X}$ that is initial among all toric morphisms into (separated) toric varieties. The toric variety $\widetilde{X}$ is unique up to unique toric isomorphism and is called the \emph{toric separation} of $X$. This construction allows to think of a toric prevariety $X$ as being constructed from a (separated) toric variety by torus-invariantly multiplying non-dense torus-orbits (e.g. in Example \ref{example_A1twoorigins} the affine line with two origins). 

 In this section toric prevarieties have been defined over $\CC$. Nevertheless all toric prevarieties naturally arise as a base change of a monoid scheme $X_\mathbb{Z}$ over $\Z$: for an affine toric variety $U_\sigma=\Spec \CC[S_\sigma]$ associated to a rational polyhedral cone $\sigma$ we have $U_{\sigma,\mathbb{Z}}=\Spec\Z[S_\sigma]$ and in general this scheme is given by gluing the $\Spec \Z[S_\sigma]$ according to the combinatorics of the system of fans defining $X$. 

\begin{definition}\label{def_otherbases}
Let $K$ an arbitrary field (or more generally an arbitrary ring $R$). A  \emph{toric prevariety} over $K$ (or $R$) is a base change to $K$ (or to $R$ respectively) of the monoid scheme $X_\mathbb{Z}$ associated to a toric prevariety $X$. 
\end{definition}

Similarly a toric morphism of toric prevarieties is also defined over $\ZZ$ and so it makes sense to refer to a base change of such a toric morphism as a \emph{toric morphism} of toric prevarieties over $K$ (or $R$) respectively. 

%%%%%%%%%%%%%%%%%%%%%%%%%%%%%%%%%%%%%%%%%%%%%%%%%%%%%%

\section{Tropical toric prevarieties}

\subsection{Tropical toric varieties}
We begin by recalling from \cite{Kajiwara_troptoric, Payne_anallimittrop} and, in particular, from \cite{Rabinoff} how to construct a tropical analogue of a toric variety. 

Write $\Rbar$ for the additive monoid $\R\cup\{\infty\}$ satisfying the rule $a+\infty =\infty$ for all $a\in\Rbar$. For a rational polyhedral cone $\sigma$ in $N_\R$ one can define the \emph{affine tropical toric variety} $U_\sigma^{trop}$ as a partial compactification of $N_\R$ by 
\begin{equation*}
U_\sigma^{trop}\deq\Hom(S_\sigma,\Rbar) \ .
\end{equation*}
The topology on $U_\sigma^{trop}$ is the weakest topology that makes all evaluation maps $u\mapsto u(s)$ for $s\in S_\sigma$ continuous. Since $S_\sigma$ is finitely generated, we may choose a finite generating set $s_1,\ldots, s_n$ of $S_\sigma$ and then $\Hom(S_\sigma,\Rbar)$ carries the subspace topology with respect to the embedding $\Hom(S_\sigma,\Rbar)\hookrightarrow \Rbar^n$ given by $u\mapsto (u(s_1),\ldots, u(s_n))$. 

For every monoid homomorphism $u\colon S_\sigma\rightarrow \Rbar$ the preimage of $\R\subseteq \Rbar$ is of the form $\tau^\perp\cap S_\sigma$ for some face $\tau$ of $\sigma$. So we naturally find a stratification 
\begin{equation}\label{eq_stratification}
\bigsqcup_{\tau \preceq \sigma} N_\R/\Span(\tau) \xlongrightarrow{\sim} U_\sigma^{trop}
\end{equation}
into locally closed subsets. Hereby an element $[u]\in N_\R/\Span(\tau)$ corresponds to the monoid homomorphism $S_\sigma\rightarrow \Rbar$ given by 
\begin{equation*}
s\longmapsto \begin{cases} u(s), & \textrm{if } s\in \tau^\perp\cap S_\sigma\\
\infty, & \textrm{else} \ 
\end{cases}
\end{equation*}
and this does not depend on the choice of representative of $[u]\in N_\R/\Span(\tau)$. 

This implies that for a face $\tau$ of $\sigma$ the monoid homomorphism $S_\sigma\rightarrow S_\tau$ induces an open embedding
\begin{equation}\label{eq_faceopenimmersion}
U_\tau^{trop}\hookrightarrow U_\sigma^{trop} \ ,
\end{equation}
whose image is the union of locally closed strata that corresponds to the faces of $\sigma$ that are also faces of $\tau$. 

\begin{definition}
Let $\Delta$ be a rational polyhedral fan in the vector space $N_\R$ spanned by the finitely generated free abelian group $N$ and denote by $X=X(\Delta)$ the associated toric variety. The \emph{tropical toric variety} associated to $\Delta$ is defined to be the colimit
\begin{equation*}
X^{trop}\deq\varinjlim_{\sigma\in \Delta} N_\R(\sigma) \ ,
\end{equation*}
taken over all cones $\sigma\in \Delta$ together with maps arising as in \eqref{eq_faceopenimmersion} from $\tau\preceq \sigma$. 
\end{definition}

\begin{example}
The tropical toric variety $\A^{2,trop}$ is given by $\Hom(\N^2,\Rbar)=\Rbar^2$, which is naturally stratified with strata $\R^2$, $\R\times\{\infty\}$, $\{\infty\}\times \R$, and $\{\infty\}\times \{\infty\}$. We may glue four copies of $\A^{2,trop}$ as indicated in Figure \ref{fig_P1xP1} to obtain the tropical toric variety $$(\PP^1\times\PP^1)^{trop}=(\R\cup\{\infty\}\cup\{-\infty\}\big)^2$$ with its nine locally closed strata.
\end{example}

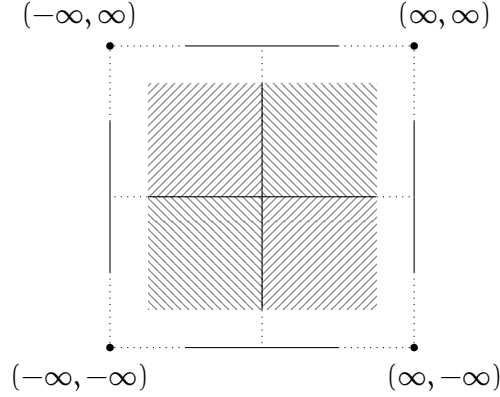
\begin{figure}[h]
\begin{tikzpicture}

\fill[pattern color=gray, pattern=north west lines] (0,0) rectangle (1.5,1.5);
\fill[pattern color=gray, pattern=north west lines] (-1.5,-1.5) rectangle (0,0);
\fill[pattern color=gray, pattern=north east lines] (-1.5,1.5) rectangle (0,0);
\fill[pattern color=gray, pattern=north east lines] (1.5,-1.5) rectangle (0,0);

\fill (2,2) circle (0.05 cm);
\fill (-2,-2) circle (0.05 cm);
\fill (2,-2) circle (0.05 cm);
\fill (-2,2) circle (0.05 cm);

\draw (0,0) -- (1.5,0);
\draw (0,0) -- (0,1.5);
\draw (2,0) -- (2,1);
\draw (0,2) -- (1,2);
\draw (0,0) -- (-1.5,0);
\draw (0,0) -- (0,-1.5);
\draw (-2,0) -- (-2,-1);
\draw (0,-2) -- (-1,-2);
\draw (-2,0) -- (-2,1);
\draw (0,-2) -- (1,-2);
\draw (2,0) -- (2,-1);
\draw (0,2) -- (-1,2);
 
\draw [dotted] (0,1) -- (0,2);
\draw [dotted] (1,0) -- (2,0);
\draw [dotted] (2,1) -- (2,2);
\draw [dotted] (1,2) -- (2,2);
\draw [dotted] (0,-1) -- (0,-2);
\draw [dotted] (-1,0) -- (-2,0);
\draw [dotted] (-2,-1) -- (-2,-2);
\draw [dotted] (-1,-2) -- (-2,-2);
\draw [dotted] (-2,1) -- (-2,2);
\draw [dotted] (-1,2) -- (-2,2);
\draw [dotted] (2,-1) -- (2,-2);
\draw [dotted] (1,-2) -- (2,-2);

\node at (2.4,2.4) {$(\infty,\infty)$};
\node at (-2.4,2.4) {$(-\infty,\infty)$};
\node at (2.4,-2.4) {$(\infty,-\infty)$};
\node at (-2.4,-2.4) {$(-\infty,-\infty)$};

\end{tikzpicture}\caption{The tropical toric variety 
$(\PP^1\times\PP^1)^{trop}$.}\label{fig_P1xP1}
\end{figure}

Let $\Delta$ and $\Delta'$ be two fans in $N_\R$ and $N_\R'$ respectively, giving rise to toric varieties $X=X(\Delta)$ and $X'=X'(\Delta')$. A homomorphism $F\colon N\rightarrow N'$ for which $F(\sigma)$ is contained in a cone in $\Delta'$ for all cone $\sigma\in \Delta$ induces a continuous map 
\begin{equation*}
F^{trop}\colon X^{trop}\longrightarrow (X')^{trop}
\end{equation*}
that respects the stratifications. 

\subsection{Tropical toric prevarieties}\label{section_troptoricpre}
We have seen in the previous section that a toric prevariety is defined by gluing toric varieties according to the combinatorial data of a system of fans. We now define a tropical toric prevariety by accordingly gluing tropical toric varieties along such data.

Let $X=X(S)$ be a toric prevariety defined by a system of fans $S=(\Delta_{ij})_{i,j\in I}$ and write $X_{ij}=X(\Delta_{ij})$ and $X_{i}=X_{ii}$. Since $\Delta_{ij}$ is a subfan of $\Delta_{ii}\cap\Delta_{jj}$, we have inclusion maps $X_{ij}\to X_{i}$ and $X_{ij}\to X_{j}$ which induce morphisms
\begin{equation*}
X_{ij}^{\trop}\longrightarrow X_{i}^{\trop}\qquad\textup{and}\qquad X_{ij}^{\trop}\longrightarrow X_{j}^{\trop}.
\end{equation*}
We now glue by the equivalence relation $\sim$ on the disjoint union
\begin{equation*}
\bigsqcup_{i\in I}X_{i}^{\trop}
\end{equation*}
defined by setting $f_{i}\sim f_{j}$ if and only if they both belong to $X_{ij}^{\trop}$ and $f_{i}=f_{j}$.

\begin{definition}
Let $S=(\Delta_{ij})_{i,j\in I}$ a system of rational polyhedral fans in $N_{\mathbb{R}}$ and let $X=X(S)$ be the corresponding toric prevariety. The \emph{tropical toric prevariety} associated to $X$ is defined to be
\begin{equation*}
X^{\trop}=\Big(\bigsqcup_{i\in I}X_{i}^{\trop}\Big)/\sim.
\end{equation*}
\end{definition}

From \eqref{eq_stratification} we immediately obtain a stratification
\begin{equation*}
    \bigsqcup_{[\sigma,i]\in\Omega(\calS)} N_\R/\Span(\sigma) \xlongrightarrow{\sim} X^{trop}
\end{equation*}
by locally closed subsets.

The construction of $X^{trop}$ is functorial with respect to toric morphisms. Indeed, by restricting a morphism $f\colon X\to X^{\prime}$ of toric prevarieties to $X_{i}$, we obtain a morphism $X_{i}\to X^{\prime}$ which is the gluing of its restrictions to open affine toric subvarieties of $X_{i}$. Since $f$ maps open affine toric subvarieties into open affine toric subvarieties, we have an induced map
\begin{equation*}
f_{i}^{\trop}\colon X_{i}^{\trop}\longrightarrow\bigsqcup_{i^{\prime}\in I^{\prime}}X_{i^{\prime}}^{\trop}
\end{equation*}
and these maps glue to a continuous map
\begin{equation*}
f^{\trop}\colon X^{\trop}\longrightarrow X^{\trop}
\end{equation*}
that respects the stratitifications.

One can also express $X^{\trop}$ as a colimit of tropical affine toric varieties as follows. For $[\sigma,i]\in\Omega(S)$, define $U_{[\sigma,i]}^{trop}=U_\sigma^{trop}=\Hom(S_\sigma,\Rbar)$. If $[\tau,j]\preceq [\sigma,i]$, then there is a map $U_{[\tau,j]}^{trop}\to U_{[\sigma,i]}^{trop}$ and they form a direct system. The maps $U_{[\sigma,i]}^{trop}\to X_{i}^{\trop}$ induce an isomorphism
\begin{equation*}
\varinjlim_{[\sigma,i]\in\Omega(S)}U_{[\sigma,i]}^{trop}\xlongrightarrow{\sim}X^{\trop}
\end{equation*}

\begin{example}\label{example_linetwooriginstrop}
Let $I=\{1,2\}$, consider the one-dimensional cone $\sigma=\mathbb{R}_{\geq 0}$, and the system of fans in $\mathbb{R}$ given by setting $\Delta_{11}=\Delta_{22}=\{\sigma,\{0\}\}$ and $\Delta_{12}=\Delta_{21}=\{\{0\}\}$ (as in Example \ref{example_A1twoorigins}). The associated tropical toric prevariety is obtained by gluing two copies of $(\mathbb{A}^{1})^{\trop}=\overline{\mathbb{R}}$ along $\mathbb{G}_{m}^{\trop}=\mathbb{R}$, thus obtaining the real line with two points at infinity (see Figure \ref{fig_A1twoorigins} below).
\end{example}

\begin{remark}\label{remark_F1}
In \cite{GiansiracusaGiansiracusa} the authors construct a tropicalization for all schemes that admit a model over $\mathbb{F}_1$, the so-called "field with one element". One may deduce from the theory in \cite{ACampoNeuenHausen_toricpre} that a toric prevariety $X$ naturally admits a model over $\mathbb{F}_1$, i.e. they arises as the base change of a monoid scheme to $\mathbb{C}$ (also see the discussion above Definition \ref{def_otherbases}). The approach of \cite{GiansiracusaGiansiracusa} would lead to a semiring scheme $X_\mathbb{T}$ over the tropical numbers, whose set of $\mathbb{T}$-valued points agrees with $X^{trop}$. 
\end{remark}

\subsection{Non-negative tropical toric prevarieties}
There is an alternative tropical analogue of a toric variety, the so-called \emph{non-negative tropicalization}, that seems to have appeared first in \cite{PopescuPampuStepanov} and naturally lends itself to generalizations to toroidal embeddings \cite{Thuillier_toroidal, ACP} and logarithmic schemes \cite{Ulirsch_functroplogsch}. 

Write $\Rbar_{\geq 0}$ for the submonoid of $\Rbar$ consisting only of non-negative elements. For a rational polyhedral cone $\sigma$ in $N_\R$ one can define the \emph{non-negative affine tropical toric variety} $U_\sigma^{trop,\geq 0}$  by 
\begin{equation*}
U_\sigma^{trop,\geq 0}\deq\Hom(S_\sigma,\Rbar_{\geq 0}) \ .
\end{equation*}
The topology on $U_\sigma^{trop,\geq 0}$ is the weakest topology that makes all evaluation maps $u\mapsto u(s)$ for $s\in S_\sigma$ continuous. Again, since $S_\sigma$ is finitely generated, we may choose a finite generating set $s_1,\ldots, s_n$ of $S_\sigma$ and then $\Hom(S_\sigma,\Rbar)$ carries the subspace topology with respect to the embedding $\Hom(S_\sigma,\Rbar_{\geq 0})\hookrightarrow \Rbar^n_{\geq 0}$ given by $u\mapsto (u(s_1),\ldots, u(s_n))$. The non-negative affine tropical toric variety $U_\sigma^{trop,\geq 0}$ is a canonical compactification of the cone $\sigma$ via the open embedding
\begin{equation}\label{eq_stratificationsnonneg}
    \sigma=\Hom(S_\sigma,\R_{\geq 0})\hooklongrightarrow \Hom(S_\sigma,\Rbar_{\geq 0})=U_\sigma^{trop,\geq 0} \ .
\end{equation}

For every monoid homomorphism $u\colon S_\sigma\rightarrow \Rbar_{\geq 0}$ the preimage of $\R_{\geq 0}\subseteq \Rbar_{\geq 0}$ is of the form $\tau^\perp\cap S_\sigma$ for some face $\tau$ of $\sigma$. So we naturally find a stratification 
\begin{equation*}
\bigsqcup_{\tau \preceq \sigma} \sigma/\tau \xlongrightarrow{\sim} U_\sigma^{trop,\geq 0}
\end{equation*}
into locally closed subsets. Here an element $[u]\in \sigma/\tau$ corresponds to the monoid homomorphism $S_\sigma\rightarrow \Rbar_{\geq 0}$ given by 
\begin{equation*}
s\longmapsto \begin{cases} u(s), & \textrm{if } s\in \tau^\perp\cap S_\sigma\\
\infty, & \textrm{else} \ 
\end{cases}
\end{equation*}
and this does not depend on the choice of representative of $[u]\in \sigma/\tau$. 

This implies that for a face $\tau$ of $\sigma$ the monoid homomorphism $S_\sigma\rightarrow S_\tau$ induces a closed embedding
\begin{equation}\label{eq_faceclosedimmersion}
U_\tau^{trop,\geq 0}\hooklongrightarrow U_\sigma^{trop,\geq 0} \ ,
\end{equation}
whose image is the closure of the face $\tau$ in $U_\sigma^{\trop,\geq 0}$. 

\begin{definition}
Let $\Delta$ be a rational polyhedral fan in the vector space $N_\R$ spanned by the finitely generated free abelian group $N$ and denote by $X=X(\Delta)$ the associated toric variety. The \emph{non-negative tropical toric variety} associated to $\Delta$ is defined to be the colimit
\begin{equation*}
X^{trop, \geq 0}\deq\varinjlim_{\sigma\in \Delta} U_\sigma^{trop,\geq 0} \ ,
\end{equation*}
taken over all cones $\sigma\in \Delta$ together with maps arising as in \eqref{eq_faceclosedimmersion} from $\tau\preceq \sigma$. 
\end{definition}

The non-negative tropical toric variety $X^{trop,\geq 0}$ naturally carries the structure of an extended cone complex in the sense of \cite{ACP, Ulirsch_functroplogsch}. From \eqref{eq_stratificationsnonneg} we immediately obtain a stratification
\begin{equation*}
    \bigsqcup_{\sigma\in\Delta} \sigma/\tau \xlongrightarrow{\sim} X^{trop,\geq 0}
\end{equation*}
by locally closed subsets.

Let $\Delta$ and $\Delta'$ be two fans in $N_\R$ and $N_\R'$ respectively, giving rise to toric varieties $X=X(\Delta)$ and $X'=X'(\Delta')$. A homomorphism $F\colon N\rightarrow N'$ for which $F(\sigma)$ is contained in a cone in $\Delta'$ for all cone $\sigma\in \Delta$ induces a continuous map 
\begin{equation*}
F^{trop,\geq 0}\colon X^{trop,\geq 0}\longrightarrow (X')^{trop,\geq 0}
\end{equation*}
that respects the stratifications. 

As in Section \ref{section_troptoricpre} above, we may define a \emph{non-negative tropical toric prevariety} by gluing according to the combinatorics of a system of fans. 

Let $X=X(S)$ be a toric prevariety defined by a system of fans $S=(\Delta_{ij})_{i,j\in I}$ and write $X_{ij}=X(\Delta_{ij})$ and $X_{i}=X_{ii}$. Since $\Delta_{ij}$ is a subfan of $\Delta_{ii}\cap\Delta_{jj}$, we have inclusion maps $X_{ij}\to X_{i}$ and $X_{ij}\to X_{j}$ which induce morphisms
\begin{equation*}
X_{ij}^{\trop,\geq 0}\longrightarrow X_{i}^{\trop,\geq 0}\qquad\textup{and}\qquad X_{ij}^{\trop,\geq 0}\longrightarrow X_{j}^{\trop,\geq 0}.
\end{equation*}
As above, we now glue by the equivalence relation $\sim$ on the disjoint union
\begin{equation*}
\bigsqcup_{i\in I}X_{i}^{\trop,\geq 0}
\end{equation*}
defined by setting $f_{i}\sim f_{j}$ if and only if they both belong to $X_{ij}^{\trop,\geq 0}$ and $f_{i}=f_{j}$.

\begin{definition}
Let $S=(\Delta_{ij})_{i,j\in I}$ a system of rational polyhedral fans in $N_{\mathbb{R}}$ and let $X=X(S)$ be the corresponding toric prevariety. The \emph{non-negative tropical toric prevariety} associated to $X$ is defined to be
\begin{equation*}
X^{\trop,\geq 0}\deq\Big(\bigsqcup_{i\in I}X_{i}^{\trop,\geq 0}\Big)/\sim.
\end{equation*}
\end{definition}

The construction of $X^{trop,\geq 0}$ is functorial. Indeed, by restricting a morphism $f\colon X\to X^{\prime}$ of toric prevarieties to $X_{i}$, we obtain a morphism $X_{i}\to X^{\prime}$ which is the gluing of its restrictions to open affine toric subvarieties of $X_{i}$. Since $f$ maps open affine toric subvarieties into open affine toric subvarieties, we have an induced map
\begin{equation*}
f_{i}^{\trop,\geq 0}\colon X_{i}^{\trop,\geq 0}\longrightarrow\bigsqcup_{i^{\prime}\in I^{\prime}}X_{i^{\prime}}^{\trop,\geq 0}
\end{equation*}
and these morphisms glue to
\begin{equation*}
f^{\trop,\geq 0}\colon X^{\trop,\geq 0}\longrightarrow X^{\trop,\geq 0}.
\end{equation*}

One can also express $X^{\trop,\geq 0}$ as a colimit of non-negative tropical affine toric varieties as follows. For $[\sigma,i]\in\Omega(S)$, define $U_{[\sigma,i]}^{trop,\geq 0}\deq U_\sigma^{\trop,\geq 0}$. If $[\tau,j]\preceq [\sigma,i]$, then there is a map $U_{[\tau,j]}^{trop,\geq 0}\to U_{[\sigma,i]}^{trop,\geq 0}$ and they form a direct system. The maps $U_{[\sigma,i]}^{\trop,\geq 0}\to X_{i}^{\trop,\geq 0}$ induce a homeomorphism
\begin{equation*}
\varinjlim_{[\sigma,i]\in\Omega(S)}U_{[\sigma,i]}^{\trop,\geq 0}\xlongrightarrow{\sim}X^{\trop,\geq 0}
\end{equation*}

\begin{example}\label{example_linewithtwooriginstropnonneg}
Let us consider again $X$ the affine line with two origins, viewed as a toric prevariety as in Example \ref{example_A1twoorigins} and Example \ref{example_linetwooriginstrop}. The corresponding non-negative tropical toric prevariety $X^{\trop,\geq 0}$ is obtained by gluing two copies of $U_{\sigma}^{\trop,\geq 0}=\overline{\mathbb{R}}_{\geq 0}$ along $U_{\{0\}}^{\trop,\geq 0}=\{0\}$ (see Figure \ref{fig_A1twoorigins} below).
\end{example}

\begin{figure}[h]
    \begin{tikzpicture}
        \draw (1,2) -- (2.5,2.375);
        \draw (1,2) -- (2.5,1.625);
        \draw[dotted] (2.5,2.375) -- (3,2.5);
        \draw[dotted] (2.5,1.625) -- (3,1.5);
        \fill (3,2.5) circle (0.05cm);
        \fill (3,1.5) circle (0.05cm);

        \draw[->] (2,1.25) -- (2,0.75);
    
        \draw (0,0) -- (2.5,0);
        \draw[dotted] (2.5,0) -- (3,0);
        \fill (3,0.5) circle (0.05cm);
        \fill (3,-0.5) circle (0.05cm);
        \fill (1,0) circle (0.05cm);
        
        \node at (1,0.4) {$0$};
        \node at (1,2.4) {$0$};
        \node at (3.4,-0.5) {$\infty_2$};
        \node at (3.4,0.5) {$\infty_1$};
        \node at (3.4,2.5) {$\infty_1$};
        \node at (3.4,1.5) {$\infty_2$};
        
    \end{tikzpicture}
    \caption{The non-negative tropical affine line $X^{trop,\geq 0}$ with two origins mapping to the tropical affine line $X^{trop}$ with two origins.}
    \label{fig_A1twoorigins}
\end{figure}
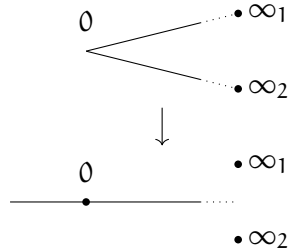

In this example, we already see that there is a natural map $X^{\trop,\geq 0}\to X^{\trop}$ which is not injective (see Proposition \ref{prop_nonnegtroptotrop} below for details).

\begin{remark}
Again, as in Remark \ref{remark_F1}, one may think of $X$ as being defined over $\mathbb{F}_1$. Taking a base change $X_{\mathcal{O}_\mathbb{T}}$ to the semiring  $\calO_{\mathbb{T}}$ of non-negative tropical numbers, we recover the non-negative tropicalization as the set of $\mathcal{O}_\mathbb{T}$-valued points of $X_{\mathcal{O}_\mathbb{T}}$. (see \cite{Lorscheid} for more details on this in the language of  \emph{blueprints}).
\end{remark}

\begin{remark}
One may think of a toric prevariety as a logarithmic scheme with respect to its boundary. When $K$ is endowed with the trivial absolute value, the tropicalization map constructed in Proposition \ref{prop_nonnegtroptoricpre} is exactly the tropicalization map for logarithmic schemes constructed in \cite{Ulirsch_functroplogsch}.
\end{remark}

\subsection{Comparison between $X^{trop}$ and $X^{trop,\geq 0}$}
For an affine toric variety $U_\sigma$ associated to a rational polyhedral cone $\sigma\subseteq N_\R$, there is a natural continuous inclusion map 
\begin{equation}\label{eq_nonnegtroptotropaffine}
    U_\sigma^{trop,\geq 0}=\Hom(S_\sigma,\Rbar_{\geq 0})\hooklongrightarrow \Hom(S_\sigma,\Rbar)=U_\sigma^{trop}
\end{equation}
that naturally embeds $U_\sigma^{trop,\geq 0}$ as the topological closure of $\sigma$ in $U_\sigma^{trop}$ and respects the stratifications.

\begin{proposition}\label{prop_nonnegtroptotrop}
Let $X$ be a toric prevariety defined by a system of fans $S=(\Delta_{ij})_{i,j\in I}$. Then there is a unique continuous map \begin{equation*}
    X^{trop,\geq 0}\longrightarrow X^{trop}
\end{equation*}
that restricts to \eqref{eq_nonnegtroptotropaffine} for all rational polyhedral cones in all $\Delta_{ij}$. This map is injective if and only if $X$ is separated and bijective if and only if $X$ is proper.
\end{proposition}

\begin{proof}
For a face $\tau$ of a rational polyhedral cone $\sigma$, the map from \eqref{eq_nonnegtroptotropaffine} naturally makes the diagram
\begin{equation*}
    \begin{tikzcd}
        U_\tau^{trop,\geq 0}\arrow[rr]\arrow[d]&& U_\tau^{trop}\arrow[d]\\
        U_\sigma^{trop,\geq 0}\arrow[rr]&& U_\sigma^{trop}
    \end{tikzcd}
\end{equation*}
commute. Therefore the locally defined map glues to an (automatically unique) continuous map $X^{trop,\geq 0}\rightarrow X^{trop}$, as desired. If $X$ is separated, it can be described by a rational polyhedral fan $\Delta$ (and not a system of fans). In this case $X^{trop,\geq 0}\rightarrow X^{trop}$ restricts to the embedding of the fan $\Delta$ on $N_\R$ and the image of $X^{trop,\geq 0}\rightarrow X^{trop}$ is precisely the closure of $\Delta$ in $X^{trop}$. This is all of $X^{trop}$ if and only the support of $\Delta$ is all of $N_\R$, which, in turn, is equivalent to $X$ being proper. 

Vice versa, suppose that $X$ is not separated. Then there are two cones $\sigma_1$ and $\sigma_2$ in two different fans in the system of fans describing $X$ that are equal as cones in $N_\R$. In $X^{trop,\geq 0}$ these two cones are glued only over a proper face, while in $X^{trop}$ those two cones would be identified in full (while their limits at infinity in $X^{trop}$ remains distinct). Thus, in this case, the map $X^{trop,\geq 0}\rightarrow X^{trop}$ is not injective.  
\end{proof}

%%%%%%%%%%%%%%%%%%%%%%%%%%%%%%%%%%%%%%%%%%%%%%%%%%%%%%

\section{The process of tropicalization}

\subsection{Analytification and tropicalization}\label{section_tropicalization}
In this section we construct a tropicalization maps for toric prevarieties that generalizes the tropicalization for toric varieties introduced in \cite{Kajiwara_troptoric} and \cite{Payne_anallimittrop} and prove some of its properties. 

The domain of this tropicalization map is the \emph{Berkovich analytification} (in the sense of \cite{Berkovich_book}). Let $X=\Spec A$ be an affine scheme of finite type over a non-Archimedean field $K$. A point $x$ in the Berkovich analytification $X^{an}$ is a multiplicative seminorm $\vert.\vert_x\colon A\rightarrow \R$ that restricts to the given non-Archimedean norm on $K$. The space $X^{an}$ carries the weakest topology making evaluation maps
\begin{equation*}
 x\longmapsto \vert f\vert_x    
\end{equation*}
for all $f\in A$ continuous. For a morphism $\phi\colon Y\rightarrow X$ of affine schemes of finite type over $K$ there is a natural continuous map $\phi^{an}\colon Y^{an}\rightarrow X^{an}$ that is given by 
\begin{equation*}
    \vert.\vert_x\longmapsto \vert.\vert_x\circ \phi^\#
\end{equation*}
where $\phi^\#\colon B\rightarrow A$ denotes the induced map on global section. In particular, when $U\subseteq X$ is an open affine subset, the induced map $U^{an}\rightarrow X^{an}$ is an open immersion. 

When $X$ is a not necessarily affine scheme that is locally of finite type over $K$, the Berkovich analytification $X^{an}$ is locally on affine open subsets given by the above space of seminorms and globally by gluing. This way every point in $X^{an}$ may be represented by an $L$-valued point for a suitable non-Archimedean extension $L$ of $K$. Moreover, for a morphism $f\colon X\rightarrow Y$ of schemes locally of finite type over $K$, there is an induced continuous map $\phi^{an}\colon X^{an}\rightarrow Y^{an}$ that restricts to the above pullbacks of seminorms on open affine subsets; the association $f\mapsto f^{an}$ is functorial in $f$. 

We refer the reader to \cite{Berkovich_book} and also to \cite{Payne_AMSnotices} for further details and, in particular, to \cite{Berkovich_book} for the definition of a structure sheaf on $X^{an}$. Since here we only care about the underlying topological space of $X^{an}$, we refrain from going further into this direction.

\begin{proposition}\label{prop_troptoricpre} Let $X$ be a toric prevariety over a non-Archimedean field $K$ given by a system of fans $S=(\Delta_{ij})_{i,j\in I}$. 
\begin{enumerate}[(i)]
\item There is a natural continuous and proper tropicalization map $\trop_X\colon X^{an}\rightarrow X^{trop}$ whose restriction to a torus-invariant open affine subset $U_{[\sigma,i]}$ for a cone $[\sigma,i]\Omega(S)$ is given by 
\begin{equation*}\begin{split}
\trop_\sigma\colon U_{[\sigma,i]}^{an}&\longrightarrow U_{[\sigma,i]}^{trop}\deq\Hom(S_\sigma,\Rbar)\subseteq X^{trop}\\
x&\longmapsto \big(s\mapsto -\log\vert \chi^s\vert_x\big) \ .
\end{split}\end{equation*}
\item For a toric morphism $f\colon X\rightarrow X'$ induced by a morphism $(F,\frakf)\colon S\rightarrow S'$ of systems of fans, the natural diagram 
\begin{equation*}\begin{tikzcd}
X^{an}\arrow[rr,"\trop_X"]\arrow[d,"f^{an}"]&& X^{trop}\arrow[d,"f^{trop}"]\\
(X')^{an}\arrow[rr,"\trop_{X'}"] &&(X')^{trop}
\end{tikzcd}\end{equation*}
commutes. 
\item There is a natural section $J_X\colon X^{trop}\rightarrow X^{an}$ of $\trop_X$ such that the composition $r_X=J_X\circ \trop_X\colon X^{an}\rightarrow X^{an}$ is a strong deformation retraction onto a closed subset $\Sigma(X^{an})$ of $X^{an}$ (that is automatically homeomorphic to $X^{trop})$.
\item Let $T^\circ=\big\{t\in T^{an}\big\vert \vert\chi^m\vert=1 \textrm{ for all } m\in M\big\}$ be the affinoid torus in $T^{an}$. There is a natural homeomorphism 
\begin{equation*}
\mu_X\colon \big\vert [X^{an}/T^\circ]\big\vert\xlongrightarrow{\sim} X^{trop,\geq 0}\end{equation*} between the underlying topological space of the non-Archimedean analytic stack quotient $[X^{an}/T^\circ]$ and $X^{trop}$ that makes the diagram 
\begin{equation*}\begin{tikzcd}
& X^{an} \arrow[rd,"\trop_X"]\arrow[ld]& \\
\big\vert {[X^{an}/T^\circ]}\big\vert\arrow[rr,"\mu_X"]&& X^{trop}
\end{tikzcd}\end{equation*}
commute. 
\end{enumerate}
\end{proposition}

\begin{proof}
Let $f\colon U_{\sigma}\to U_{\sigma^{\prime}}$ be a toric morphism of affine toric varieties and write $f^{\#}$ for the induced $K$-algebra homomorphism  $K[S_{\sigma^{\prime}}]\to K[S_{\sigma}]$ and the monoid homomorphism $S_{\sigma^{\prime}}\rightarrow S_\sigma$. Let $x\in X^{an}$ and $s^{\prime}\in S_{\sigma^{\prime}}$. Then we have
\begin{equation*}\begin{split}
(\trop_{U_{\sigma^{\prime}}}\circ\:f^{an})(x)(s^{\prime})&=-\log\vert \chi^{s^{\prime}}\vert_{f(x)}\\&=-\log\vert f^{\#}(\chi^{s^{\prime}})\vert_{x}\\
&=-\log\vert\chi^{f^\#(s)}\vert_{x}\\&=\trop_{U_\sigma}(x)(f^\#(s^{\prime}))
=(f^{trop}\circ \trop_{U_\sigma})(x)(s^{\prime})
\end{split}\end{equation*}
and this shows that the diagram
\begin{equation}\label{eq_commutavityaffinetrop}
    \begin{tikzcd}
        U_\sigma^{an}\arrow[d,"f^{an}"]\arrow[rr,"\trop_{U_\sigma}"]&& U_\sigma^{trop}\arrow[d,"f^{trop}"]\\
        U_{\sigma^{\prime}}^{an}\arrow[rr,"\trop_{U_{\sigma^{\prime}}}"]  && U_{\sigma^{\prime}}^{trop}
    \end{tikzcd}
\end{equation}
commutes. 

For Part (i) let $\sigma$ be a rational polyhedral cone. The commutativity of \eqref{eq_commutavityaffinetrop} applied to the open embedding of a torus-invariant open subset $U_\tau\subseteq U_\sigma$ associated to a face $\tau$ of $\sigma$ shows that $\trop_{U_\sigma}$ naturally restricts to $\trop_{U_\tau}$ on $U_\tau^{an}\subseteq U_\sigma^{an}$. The existence of $\trop_X$ now follows, since we can write $X^{an}$ as the colimit of the $U_{[\sigma,i]}^{an}$ for all torus-invariant open affine subsets $U_{[\sigma,i]}$ of $X$. The continuity of $\trop_X$ follows from the continuity of the $\trop_{U_{[\sigma,i]}}$ and there it is obvious. 

Let $A$ be a compact subset of $X^{trop}$, i.e. $A$ is a quasi-compact Hausdorff space. Then $A$ has to lie in a Hausdorff open subset in $X^{trop}$ of the form $X(\Delta_{i})$ for a fan $\Delta_i$ in $\calS$ (here we choose the fans in $\calS$ to be maximal). Then $\trop_X^{-1}(A)\subseteq X(\Delta_i)^{an}$ and so it is enough to show the properness of $\trop_X$ for a (separated) toric variety. Here this is a fact that seems to be well-known in the community. We only provide a proof, since we could not find it written explicitly in the literature. 

Let $X=X(\Delta)$ be a toric variety associated to a fan $\Delta$ in the vector space $N_\R$ generated by the cocharacter lattice $N$ of its big torus $T$. Let $A\subseteq X^{trop}$ a compact subset. Choose a toric completion $\widetilde{X}$ of $X$. Then $A$ is a closed subset of $\widetilde{X}^{trop}$. Its preimage $\trop_{\widetilde{X}}^{-1}(A)$ is closed, since $\trop_{\widetilde{X}}$ is continuous; so, using that $\widetilde{X}^{trop}$ is compact, we find that 
\begin{equation*}
    \trop_X^{-1}(A)=\trop_{\widetilde{X}}^{-1}(A)
\end{equation*}
is compact as well.

For Part (ii) it suffices to check this statement for an torus-invariant open affine subset $U_{\sigma}\subseteq X$ and then this is the commutativity of \eqref{eq_commutavityaffinetrop}.

In Part (iii) we again first consider a torus-invariant open affine subset $U_\sigma$ of $X$. Here the section $J_{U_\sigma}\colon U_\sigma^{trop}\rightarrow U_\sigma^{an}$ is given by associating to $u\in U_\sigma=\Hom(S_\sigma, \Rbar)$ the multiplicative seminorm defined by
\begin{equation*}
    \Big\vert \sum_{s\in S_\sigma}a_s\chi^s\Big\vert_{J(u)}\deq\sup_{s\in S_\sigma}\big\{\vert a_s \vert e^{-u(s)}\big\}
\end{equation*}
for an element $\sum_{s\in S_\sigma}a_s\chi^s\in K[S_\sigma]$. It is clear that this map is continuous and that the equality $\trop_{U_\sigma}\circ J_{U_\sigma}=\id_{U_\sigma}$ holds. Set $\rho_{U_\sigma}:=J_{U_\sigma}\circ \trop_{U_\sigma}$ and observe that 
\begin{equation*}
    \rho_{U_\sigma}\circ\rho_{U_\sigma}=J_{U_\sigma}\circ \trop_{U_\sigma}\circ J_{U_\sigma}\circ \trop_{U_\sigma}=J_{U_\sigma}\circ \trop_{U_\sigma}=\rho_{U_\sigma} \ . 
\end{equation*}
Therefore $\rho_{U_\sigma}\colon X^{an}\rightarrow X^{an}$ is a retraction onto a closed subset of $X^{an}$, the \emph{non-Archimedean skeleton} of $U_\sigma^{an}$ (which is automatically homeomorphic to $U_\sigma^{trop}$). We refer the reader to \cite[Section 2]{Thuillier_toroidal} for an explicit construction of a strong homotopy between the identity map on $U_\sigma^{an}$ and $\rho_{U_\sigma}$. 

Since all of these construction are naturally compatible with restrictions to a torus-invariant open affine subset $U_\tau$ of $U_\sigma$ for a face $\tau$ of $\sigma$, there is also a continuous section $J_X\colon X^{trop}\rightarrow X^{an}$ of $\trop_X$, a retraction $\rho_X\colon X^{an}\rightarrow X^{an}$ given by $\rho_X:=J_X\circ \trop_X$, and a strong homotopy between the identity on $X^{an}$ and $\rho_X$.

In Part (iv) we may again reduce to a torus-invariant open affine subset $U_\sigma$ and in this case the proof is identical to the one presented in \cite{Ulirsch_tropisquot}. 
\end{proof}

%%%%%%%%%%%%%%%%%%%%%%%%%%%%%%%%%%%%%%%%%%%%%%%%%%%%%%

\subsection{Raynaud generic fiber and non-negative tropicalization}\label{section_positivetropicalization}
In this section we construct a natural \emph{non-negative tropicalization map}, which seems to have appeared first in \cite{PopescuPampuStepanov} for affine toric varieties.

The domain of the non-negative tropicalization map is not the Berkovich analytification, but rather a certain variant called the \emph{Raynaud generic fiber} that, in its simplest form, associates to a scheme locally of finite type over the valuation ring $R$ of non-Archimedean field $K$ a Berkovich analytic space $X^\circ$ (see \cite[Sections 5.2 and 5.3]{Temkin_introBerkovich} for details). 

On an affine $R$-scheme $X=\Spec A$ of finite type the underlying topological space of the Raynaud generic fiber $X^\circ$ is the closed subset of $X_K^{an}$ that parametrizes multiplicative seminorms $\vert.\vert_x\colon A\rightarrow R_{\geq 0}$ that are \emph{bounded}, i.e. that fulfill $\vert f\vert_x\leq 1$ for all $f\in A$. Given a morphism $\phi\colon Y\rightarrow X$ of affine schemes $X=\Spec A$ and $Y=\Spec B$ of finite type over $R$ the induced map $\phi^{an}\colon Y_K^{an}\rightarrow X_K^{an}$ restricts to a natural continuous map $\phi^\circ\colon Y^\circ\rightarrow X^\circ$. Notice, however, that for an open affine subset $U=\Spec B$ of $X=\Spec A$, the subset $U^\circ\subseteq X^\circ$ is closed (and, in fact, a so-called \emph{affinoid domain}).

For a not necessarily affine scheme $X$ that is locally of finite type over $R$, the Raynaud generic fiber is the colimit of all $U^\circ$ for open affine subsets $U=\Spec A$ of $X$. Note that this way every point of $X^\circ$ may be represented by an $R'$-valued point for the valuation ring $R
$ of a non-Archimedean extension of $K$. Also every morphism $f\colon X\rightarrow Y$ of schemes locally of finite type over $R$ induces a natural continuous map $f^\circ \colon X^\circ\rightarrow Y^\circ$ that restricts to the pullback of seminorms on open affine subsets; the association $f\mapsto f^\circ$ is functorial.

There is a natural continuous map $X^\circ\rightarrow X^{an}$ that is locally induced by the inclusions $U^\circ\hookrightarrow U^{an}$ on open affine subsets $U=\Spec A$ of $X$. By the valuative criteria, this map is injective if and only $X$ is separated and bijective if and only if $X$ is proper. 

\begin{proposition}\label{prop_nonnegtroptoricpre} Let $X$ be a toric prevariety over a complete valuation ring $R$ given by a system of fans $S=(\Delta_{ij})_{i,j\in I}$. 
\begin{enumerate}[(i)]
\item There is a natural continuous and proper tropicalization map $\trop_X\colon X_\eta\rightarrow X^{trop,\geq 0}$ whose restriction to a torus-invariant open affine subset $U_{[\sigma,i]}$ for a cone $[\sigma,i]\in\Omega(S)$ is given by 
\begin{equation*}\begin{split}
\trop_\sigma\colon U_{[\sigma,i]}^\circ&\longrightarrow U_{[\sigma,i]}^{trop,\geq 0}\deq\Hom(S_\sigma,\Rbar_{\geq 0})\\
x&\longmapsto \big(s\mapsto -\log\vert \chi^s\vert_x\big) \ .
\end{split}\end{equation*}
\item For a toric morphism $f\colon X\rightarrow X'$ induced by a morphism $(F,\frakf)\colon S\rightarrow S'$ of systems of fans, the natural diagram 
\begin{equation*}\begin{tikzcd}
X^\circ\arrow[rr,"\trop_X"]\arrow[d,"f^\circ"]&& X^{trop,\geq 0}\arrow[d,"f^{trop}"]\\
(X')^\circ\arrow[rr,"\trop_{X'}"] && (X')^{trop,\geq 0}
\end{tikzcd}\end{equation*}
commutes. 
\item There is a natural section $J_X\colon X^{trop,\geq 0}\rightarrow X^\circ$ of $\trop_X$ such that the composition $r_X=J_X\circ \trop_X\colon X^\circ\rightarrow X^\circ$ is a strong deformation retraction onto a closed subset $\Sigma(X^\circ)$ of $X^\circ$ (which is automatically homeomorphic to $X^{trop,\geq 0}$).
\item Let $T^\circ=\big\{t\in T^{an}\big\vert \vert\chi^m\vert=1 \textrm{ for all } m\in M\big\}$ be the affinoid torus in $T^{an}$. There is a natural homeomorphism 
\begin{equation*}
\mu_X\colon \big\vert [X^{\circ}/T^\circ]\big\vert\xlongrightarrow{\sim} X^{trop,\geq 0}
\end{equation*} between the underlying topological space of the non-Archimedean analytic stack quotient $[X^{\circ}/T^\circ]$ and $X^{trop,\geq 0}$ that makes the diagram 
\begin{equation*}\begin{tikzcd}
& X^{\circ} \arrow[rd,"\trop_X"]\arrow[ld]& \\
\big\vert {[X^{\circ}/T^\circ]}\big\vert\arrow[rr,"\mu_X"]&& X^{trop,\geq 0}
\end{tikzcd}\end{equation*}
commute. 
\end{enumerate}
\end{proposition}

\begin{proof}
Let $f\colon U_\sigma\rightarrow U_\sigma$ be a toric morphism between two affine toric varieties associated to the rational polyhedral cones $\sigma$ and $\sigma'$. Then the commutativity of \eqref{eq_commutavityaffinetrop} implies the commutativity of \begin{equation}\label{eq_commutavityaffinenonnegtrop}
    \begin{tikzcd}
        U_\sigma^{\circ}\arrow[d,"f^{\circ}"']\arrow[rr,"\trop_{U_\sigma}"]&& U_\sigma^{trop, \geq 0}\arrow[d,"f^{trop}"]\\
        U_{\sigma^{\prime}}^{\circ} \arrow[rr,"\trop_{U_{\sigma^{\prime}}}"]  && U_{\sigma^{\prime}}^{trop,\geq 0}
    \end{tikzcd}
\end{equation}
One now argues as in the proof of Proposition \ref{prop_troptoricpre} in order to show that Parts (i)-(iii). We point out the proof of the properness of $\trop_X$ for a (separated) toric variety $X$ is easier here, since both $X^\circ$ and $X^{trop}$ are already compact. 

For Part (iv) we again reduce to the case of a torus-invariant open affine subset $U_\sigma$ of $X$. We recall from \cite{Ulirsch_tropisquot} that the central part of the argument was that $U_\sigma^{trop}$ is the topological $1$-colimit of the groupoid presentation $T_K^\circ\times U_{\sigma, K}^{an}\rightrightarrows U_{\sigma,K}^{an}$ of $\big[U_{\sigma,K}^{an}/T^\circ\big]$. But this implies that $U_\sigma^{trop,\geq 0}$ is the topological $1$-colimit of the groupoid presentation $T^\circ\times X^\circ\rightrightarrows X^\circ$ and so we find that $X^{trop,\geq 0}$ is naturally homeomorphic to the underlying topological space $\big\vert\big[ X^\circ/T^\circ\big]\big\vert$ of $\big[ X^\circ/T^\circ\big]$ (in the sense of \cite[Section 3]{Ulirsch_tropisquot}).
\end{proof}

The non-negative tropicalization map is compatible with the usual tropicalization map, as indicated in the following.

\begin{proposition}\label{prop_compatibility}
Let $X=X(S)$ be a toric prevariety over the valuation ring $R$ of a non-Archimedean field $K$ and write $X_K\deq X\times_R K$ for its scheme-theoretic generic fiber. Then the natural diagram
\begin{equation*}
\begin{tikzcd}
X^{\circ}\arrow[rr,"\trop_X^{\geq 0}"]\arrow[d]&&X^{trop,\geq 0}\arrow[d]\\
X_K^{an}\arrow[rr,"\trop_{X_K}"]&& X^{trop}
\end{tikzcd}
\end{equation*}
commutes.
\end{proposition}

So, when $X$ is separated, the non-negative tropicalization map is precisely the restriction of the usual tropicalization map to $X^\circ\subseteq X^{an}$.

\begin{proof}[Proof of Proposition \ref{prop_compatibility}]
We may check the commutativity of the diagram on open affine torus-invariant subsets $U_{[\sigma,]}\subseteq X$ for a rational polyhedral cone $[\sigma,i]\in\Omega(S)$. But in this case we have a commutative diagram
\begin{equation*}
    \begin{tikzcd}
        U_{[\sigma,i]}^\circ \arrow[rr,"\trop_{U_{[\sigma,i]}}^{\geq 0}"]\arrow[d]&& U_{[\sigma,i]}^{trop,\geq 0}\arrow[d]\\
        U_{[\sigma,i],K}^{an}\arrow[rr,"\trop_{U_{[\sigma,i],K}}"] && U_{[\sigma,i]}^{trop}\ ,
    \end{tikzcd}
\end{equation*}
since the formulas in Proposition \ref{prop_troptoricpre} (i) and Proposition \ref{prop_nonnegtroptoricpre} (i) agree for $x\in U_{[\sigma,i]}^{\circ}\subseteq U_{[\sigma,i]}^{an}$.
\end{proof}

\begin{remark}
When $K$ is carrying the trivial absolute value, its valuation ring is simply $K$ itself and the Raynaud generic fiber of a scheme $X$ that is locally of finite type over $K$ is nothing but the $\beth$-space constructed in \cite{Thuillier_toroidal}. The non-negative tropicalization map then turns out to be a special case of the tropicalization map constructed in \cite{Ulirsch_functroplogsch} for a fine and saturated logarithmic scheme $X$. This tropicalization map, in turn, can be identified with a retraction to a non-Archimedean skeleton of $X^\beth$, when $X$ is logarithmically smooth over $K$ by \cite{Thuillier_toroidal}. 

\end{remark}

%%%%%%%%%%%%%%%%%%%%%%%%%%%%%%%%%%%%%%%%%%%%%%%%%%%%%%

\section{Homogeneous spectra of multigraded rings}

In this section we show how to associate a reasonably well-behaved scheme to a ring $S$ graded by a finitely generated abelian group $D$. The construction was invented by Brenner and Schr\"oer \cite{BrennerSchroer}, and it generalizes the usual projective spectrum of an $\N$-graded ring. A more detailed discussion of this material (in particular an analysis of the connection with convex geometry) can be found in \cite{KU}. Here we will only focus on the construction itself, and the path leading to embeddings of divisorial schemes into simplicial toric prevarietiees. This section is expository. 

For the duration of this section let $D$ be a finitely generated abelian group, and let 
\[
S \equ \bigoplus_{d\in D} S_d
\]
be a $D$-graded ring. We will write $S_d$ for the degree $d$ homogeneous piece of $S$, and $\deg_D(f)$ for the degree of a homogeneous element $f\in S$. 

\begin{definition}[Relevant element]\label{defn:relevant}
We call a $D$-graded ring $S$ \emph{periodic} if the subgroup
\[
\big\{ \deg_D(f) \mid f\in S^\times \text{ homogeneous }\big\} \leq D
\]
is of finite index (here $S^\times$ denotes the set of homogeneous units in $S$). An element $f$ of a $D$-graded ring $S$ is said to be \emph{relevant} if it is homogeneous, and the localization $S_f$ is periodic. We denote the set of relevant elements of $S$ by $\Rel_D(S)$. 	
\end{definition}

\begin{remark}
If $D\simeq \Z$, then every homogeneous element of non-zero degree is relevant. 
\end{remark}

\begin{lemma}[\cite{BrennerSchroer} Lemma 2.1]\label{lem:periodic git}
	Let $S$ be a periodic $D$-graded ring. Then the projection morphism $\Spec (S)\to \Spec (S_0)$ induced by the inclusion $S_0\hookrightarrow S$
	is a geometric quotient in the sense of geometric invariant theory.  In particular, for $f\in\Rel^D(S)$, the morphism $S_f\to S_{(f)} \deq (S_f)_0$ will be a geometric quotient.  
\end{lemma}
   
As is well-known, a grading of $S$ by a finitely generated abelian group $D$ corresponds to the action of the diagonalizable group scheme $\Spec (S_0[D])$ on $\Spec S$. If $S$ is an algebra over a field then geometric invariant theory \cite[Theorem 1.1]{MFK} tells us that the projection morphism $\Spec (S) \to \Spec (S_0)$ is a categorical quotient in the category of schemes. Note, however, that the latter space is often very simplistic (e.g. consider the case of a polynomial ring over a field graded by the degree) hence not what we want in general. 
	
Brenner and Schr\"oer take a different route: $\Spec (S)$ has  a quotient 
\[
\Spec (S) \longrightarrow \Quot(S)
\]	
in the category of ringed spaces (see \cite[Exercise 2.14]{Liu}), which is often quite different from the quotient in the category of schemes. The ringed space $\Quot(S)$ will serve as an 'ambient space' for the construction of 
of the space we are looking for.  It is important to point out that  the construction of \cite{BrennerSchroer} works over arbitrary rings, and not only for algebras over fields as in \cite{MFK}.  Localization by relevant elements yields periodic  rings where the two quotients (in the category of schemes and in the category of ringed spaces) agree. 

Homogeneous localization by relevant elements yields open sub-ringed-spaces of $\Quot(S)$ 
\[
D_+(f) \deq \Spec (S_{(f)}) \dsubseteq \Quot(S)
\]
that are schemes themselves. This in turn yields the following construction. 

\begin{definition}[Homogeneous spectrum of a multigraded ring]\label{defn:multihomogeneous spectrum}
Let $D$ be a finitely generated abelian group, and let $S$ be a $D$-graded ring. We define the \emph{(multi)homogeneous spectrum of $S$} as the scheme 
\[
\Proj_D S \deq \bigcup_{f\in S \text{ relevant}} D_+(f) \dsubseteq \Quot(S)\ .
\]   
\end{definition}

As one would expect, we can define  the irrelevant ideal of $S$ via
\[
S_+ \deq ( f\in S\mid \text{$f$ is relevant}) \lhd S\ ,
\]
and call $V(S_+)\subseteq \Spec  (S)$ the irrelevant subscheme. This of course depends on $D$ to a large extent. 

One then obtains an affine projection morphism
\[
\Spec (S) \setminus V(S_+) \longrightarrow \Proj_D(S)\ ,
\]
which is a geometric quotient for the induced action. 

The points of $\Proj_D (S)$ correspond to graded (not necessarily prime!) ideals $\mathfrak{p} \lhd S$ not containing $S_+$ and such that the subset of homogeneous elements $H\subseteq S\setminus \mathfrak{p}$ is closed under multiplication. 

\begin{remark}
If $S$ is an integral domain then $\Proj_D S$ is an integral scheme by \cite[Lemma 3.6]{KU}.
\end{remark}

The following example is taken from \cite{KU}. It illustrates the strong dependence of $\Proj_D(S)$ on the grading. 

\begin{example}
Let $k$ be a field. We will compute  multihomogeneous spectra of the ring $S=k[T_1,T_2]$ for various gradings. This  example serves to illustrate that regradings have a profound effect on the geometry of $\Proj^DS$ (including separatedness and dimension), and highlights the contrast to the classical $\N$-graded spectrum. 
\begin{enumerate}
	\item First let $D=\Z^2$, and grade $S$ by degree, that is set $S_{(a,b)}=k\cdot T_1^aT_2^b$ for $(a,b)\in \N^2$. Then homogeneous elements are constant multiples of monomials, and 
	\[
	\Rel_D(S) \equ \{ \alpha T_1^aT_2^b\mid a,b\geq 1\, ,\, \alpha\in k\}\ .
	\]
	As a consequence, we obtain $S_{(f)}\simeq k$ for any $f\in\Rel_D(S)$, and $\Proj_DS$ turns out to be a point. 
	\item Let $D=\Z^2$ again, but now 
	\[
	S_{(a,b)} \deq \begin{cases}
	\bigoplus_{c,d\geq 0,c+d=a} k\cdot T_1^aT_2^b & \text{ if $b=0$} \\ 0 & \text{ otherwise.}
	\end{cases}
	\] 
	This is the grading induced by the (non-surjective) homomorphism $\delta\colon \Z^2\to \Z^2\, ,\, (c,d)\mapsto (c+d,0)$. Since all degrees lie on the $a$-axis, no homogeneous element is relevant, hence $\Proj_DS=\emptyset$. 
	\item For the sake of completeness let now $D=\Z$, and consider the grading of $S$ by total degree. Essentially by definition $\Proj_DS$ is the usual $\N$-graded projective spectrum of $S$, that is, $\Proj_DS\simeq \PP^1$. The grading in this case can be realized via the surjective homomorphism $\delta\colon \Z^2\to \Z$ given by $(a,b)\mapsto a+b$. 
	 \item Last, we look at the surjective regrading $\delta\colon \Z^2\to \Z$ via $(a,b)\mapsto a-b$. As explained in \cite[Beginning of Section 3]{BrennerSchroer}, we have that $\Proj_DS$ is isomorphic to the affine line with double origin. 
\end{enumerate}
\end{example} 

As we saw above, one issue with the construction is that $\Proj_D(S)$ might not be separated if $\rk D\geq 2$. It is in general an interesting question what conditions on $D$ and $S$ would guarantee separability. We will not pursue this here, in fact, the lack of separability is a feature that enables us to study prevarieties as locally closed subschems of toric prevarieties.

\begin{example}
Let $S=R[T_1,\ldots, T_r]$ be a polynomial ring over a commutative ring and suppose that it is graded via a surjective group homomorphism $\phi\colon\Z^r\rightarrow \Z^m$. Then $\Proj_{\Z^m}S$ is a simplicial toric prevariety by \cite[Proposition 3.4]{BrennerSchroer}. The discussion in \cite[Section 8]{ACampoNeuenHausen_toricpre} shows that every simplicial toric prevariety arises in this fashion.
\end{example}

We now  move on to describing morphisms and rational maps to projective spectra. The following is an elaboration of \cite[Proposition 4.1 and Proposition 4.2]{BrennerSchroer},  see also \cite[Section 2]{KU}.
Let $D$ be a finitely generated abelian group, let $X$ be a scheme, and let 
\[
\calB \equ \bigoplus_{d\in D} \calB_d
\]
be a $D$-graded quasicoherent $\calO_X$-algebra. 

\begin{definition}
With notation as above, let  $s\in\Gamma(X,\calB_d)$ be a homogeneous section of degree $d$. We write $X_s\subseteq X$ for the largest open subset such that the multiplication maps
\begin{eqnarray*}
f^m \colon \calO_X & \longrightarrow & \calB_{md} \\
t & \longmapsto & ts^m
\end{eqnarray*}
are bijective for all $m\geq 0$. 
\end{definition}

For our purposes we will define a rational map between schemes as a morphism from a non-empty open subset of the source to the target.

\begin{proposition}\label{prop: rtl maps into hom spectra}
Let $D$ be a finitely generated abelian group and $S$ a $D$-graded ring. Let $X$ be an arbitrary scheme. The choice of a $D$-graded quasicoherent $\calO_X$-algebra $\calB$ and a $D$-graded ring homomorphism
\[
\phi\colon S \longrightarrow \Gamma(X,\calB)
\]
yields a rational map $X\rtl \Proj_D(S)$ in the following way. 

Write 
\[
U_{\calB,\phi} \deq \bigcup_{f\in\Rel_D(S)} X_{\phi(f)}\ .
\]
Then 
\begin{enumerate}
    \item the subset $U_{\calB,\phi}\subseteq X$ is non-empty and open;
    \item there exists a natural morphism of schemes $r_{\calB,\phi}\colon U_{\calB,\phi}\to \Proj (S)$ along with a commutative diagram
\[
\xymatrix{
U_{\calB,\phi} \ar[d]^{r_{\calB,\phi}}   & \Spec (\calB|_{U_{\calB,\phi}}) \ar[l] \ar[r] \ar[d]^{\phi^{\#}} & \Spec(\calB) \ar[d]^{\phi^{\#}}  \\
\Proj(S) & \Spec(S) \setminus V_+ \ar[l] \ar[r] & \Spec(S)
}
\]
The unnamed morphisms  pointing to the left are natural projections, those pointing to the right are natural inclusion morphisms.
\end{enumerate}
\end{proposition}

\begin{proof}
Let $f\in \Rel_D(S)$. Then the $D$-graded homomorphism $\phi\colon S\to \Gamma(X,\calB)$  yields a $D$-graded  homomorphism $\phi_{f}\colon S_f\to \Gamma(X,\calB)_{\phi(f)}$, which descends to a homomorphism 
\[
\phi_{(f)} \colon S_{(f)} \longrightarrow \Gamma(X,\calB)_{(\phi(f))}\ .
\]
Taking the associated morphisms of affine schemes gives rise to the commutative diagram 
\[
\xymatrix{
\Spec \Gamma(X,\calB)_{(\phi(f))} \ar[d]^{\phi_{(f)}^{\#} }   & \Spec \Gamma(X,\calB)_{\phi(f)}  \ar[l] \ar[r] \ar[d]^{\phi^{\#}} & \Spec(\calB) \ar[d]^{\phi^{\#}}  \\
\Spec S_{(f)}  & \Spec(S_f) \ar[l] \ar[r] & \Spec(S)
}\ .
\] 
Since
\[
\Spec \Gamma(X,\calB)_{\phi(f)}  \simeq \Spec \Gamma(X_{\phi(f)},\calB)\ ,
\]
the square on the right-hand side in the Proposition commutes by the definition of $U_{\calB,\phi}$. 

Next, by the definition of $X_{\phi(f)}$, the multiplication maps 
\[
\phi(f)^m|_{X_{\phi(f)}} \colon \calO_X|_{X_{\phi(f)}} \longrightarrow \calB_{md}|_{X_{\phi(f)}}
\]
are bijective for all natural numbers $m$. This way we obtain a  homomorphism 
\begin{eqnarray*}
\rho_f \colon \Gamma(X,\calB)_{(\phi(f))} & \longrightarrow & \Gamma(X_{\phi(f)},\calO_X) \\ 
\frac{g}{f^m} & \longmapsto & \left( \phi(f)^m|_{X_{\phi(f)}} \right)^{-1}(g)\ .
\end{eqnarray*}
The composition
\[
S_{(f)} \stackrel{\phi_{(f)}}{\longrightarrow } \Gamma(X,\calB)_{(\phi(f))} \stackrel{\rho_f}{\longrightarrow} 
\Gamma(X_{\phi(f)},\calO_X)
\]
yields a morphism of schemes
\[
X_{\phi(f)} \simeq \Spec \Gamma(X_{\phi(f)},\calO_X) \longrightarrow \Spec S_{(f)} \equ D_+(f) \dsubseteq \Proj_D(S)\ . 
\]
These morphisms are compatible with intersections therefore they give rise to a morphism 
\[
U_{\calB,\phi} \equ \bigcup_{f\in\Rel_D(S)} X_{\phi(f)} \longrightarrow \Proj_D(S)\ .
\]
Finally, the morphism 
\[
\Spec ( \calB|_{U_{\calB,\phi}} ) \longrightarrow U_{\calB,\phi}
\]
in the diagram arises from the ring homomorphisms
\[
\Gamma(X_{\phi(f)},\calO_X) \longrightarrow \Gamma(X_{\phi(f)},\calB) \simeq \Gamma(X,\calB)_{\phi(f)}\ .
\]\end{proof}

%%%%%%%%%%%%%%%%%%%%%%%%%%%%%%%%%%%%%%%%%%%%%%%%%%%%%

\section{Divisorial schemes and prevarieties}

The notion of a projective variety has certainly a number of  desirable properties, but it does not suffice in a number of situations, especially when it comes moduli theory. Therefore, relaxations of the concept are of interest. Traditionally one can consider quasi-projective varieties or projective schemes, both of which keep the requirement of the existence of an ample invertible sheaf. Here we recall a more recent variant due to Borelli \cite{Borelli_divisorialvarieties,BrennerSchroer} which asks for a lighter version of this, still maintaining most of the useful properties of projective schemes. 

At the heart of the definition lies the following generalization of Grauert's criterion for ample sheaves.  Let $L_1,\dots,L_m$ be invertible sheaves on a scheme $X$. For a multiindex $d=(d_1,\dots,d_m)\in\ZZ^m$ we write
\[
L^{\otimes d} \deq L_1^{\otimes d_1}\otimes\ldots\otimes L_m^{\otimes d_m}\ .
\]

\begin{proposition}\label{prop:Grauert for ample families}
Let $X$ be a qcqs (quasi-compact and quasi-separated) scheme, $L_1,\dots,L_m$ a collection of invertible sheaves on $X$. Then the following are equivalent:
\begin{enumerate}
    \item The open sets $X_f$ with $f\in \Gamma(X,L^{\otimes d})$ for all $d\in\NN^m$ form a base of the topology of $X$. 
    \item For every point $x\in X$ there exists $d\in \NN^m$ and a global section $f\in \Gamma(X,L^{\otimes d})$ such that $X_f$ is an affine open neighbourhood of $x$. 
    \item For every point $x\in X$ there exists a $\QQ$-basis $d^{(1)},\dots,d^{(m)}$ of $\NN^m$ along with global sections $f_i\in\Gamma(X,L^{\otimes d^{(i)}})$ such that the open subsets $X_{f_i}$ are affine neighbourhoods of $x$
 \end{enumerate}
 \end{proposition}
 
We refer the reader to \cite[Proposition 1.1]{BrennerSchroer} for a proof.
 
\begin{definition}[Ample families and divisorial schemes]
Let $X$ be a qcqs (quasi-compact and quasi-separated) scheme, $L_1,\dots,L_m$ a collection of invertible sheaves on $X$. If the invertible sheaves $L_1,\dots,L_m$ satisfy the equivalent conditions of Proposition~\ref{prop:Grauert for ample families} then we call them an \emph{ample family} or an \emph{ample system}. A qcqs scheme is called \emph{divisorial} if it admits an ample family of invertible sheaves. 
\end{definition} 
 
\begin{proposition}\label{prop:restriction}
Let $X$ be a divisorial scheme, $L_1,\dots,L_r$ an ample system on $X$, and let $Y\subseteq X$ be a locally closed subscheme. Then the restriction $L_1\vert_Y,\dots,L_r\vert_Y$ forms an ample system on $Y$, in particular, $Y$ is a divisorial scheme as well. 
\end{proposition} 
\begin{proof}
We will need to prove that $Y$ is a divisorial scheme, that is, that $L_1\vert_Y,\dots,L_r\vert_Y$ is an ample system.  

We will prove that the collection of open subsets
\[
\big\{ X_s\mid s\in \Gamma(Y,(L|_Y)^d\ \text{for all $d\in\N$}\big\} 
\]
forms a base of the Zariski topology on $Y$. To this end, let $V\subseteq Y$ be an open subset of $X$. Since the Zariski topology of $Y$ is the subspace topology of the Zariski topology of $X$, there exists an open subset $U\subseteq X$ such that $V = U\cap Y$. As $L_1,\dots,L_r$ form an ample family, there exists an index set $I$, multiindices $d_i\in\N^r$,  and global sections $s_i\in\Gamma(X,L^{\otimes d_i})$ such that
\[
U \equ \bigcup_{i\in I} X_{s_i}\ ,
\]
and consequently, 
\[
V \equ U \cap Y \equ \bigcup_{i\in I} X_{s_i}\cap Y\ .
\]
By Lemma \ref{lem:Borelli 2.3} below, 
\[
X_{s_i} \cap Y \equ Y_{s_i|_Y}\ ,
\]
therefore we are done. 
\end{proof}
 
\begin{lemma}\label{lem:Borelli 2.3}
Let $X$ be an arbitrary scheme, let $L$ be an invertible sheaf on $X$, and let $Y\subseteq X$ be a locally closed subscheme. For every $s\in \Gamma(X,L)$
\[
X_s\cap Y \equ Y_{s|_Y}\ .
\]
\end{lemma} 
\begin{proof}
This is a part of \cite[Proposition 2.3]{Borelli_divisorialvarieties}. We point out that his argument (along with those for \cite[Proposition 2.1 and Proposition 2.2]{Borelli_divisorialvarieties}) remain valid over an arbitrary scheme. 
\end{proof}

\begin{lemma}\label{lem:Xf relevant affine}
Le $X$ be a qcqs scheme and let $L_1,\dots,L_r$ be an ample family. The collection
\[
\big\{X_f\mid \text{$f\in\Gamma(X,\oplus_{d\in\N^r}L^{\otimes d})$ is relevant and $X_f$ is affine}\big\}
\]
forms an open cover of $X$. 
\end{lemma} 
\begin{proof}
Let $x\in X$ be an arbitrary point. By Proposition~\ref{prop:Grauert for ample families} (3) there exists  a $\QQ$-basis $d^{(1)},\dots,d^{(m)}$ of $\NN^m$ along with global sections $f_i\in\Gamma(X,L^{\otimes d^{(i)}})$ such that the open subsets $X_{f_i}$ are affine neighbourhoods of $x$. Consider 
\[
f \deq f_1\cdot\ldots\cdot f_m\ .
\]
Then $f$ is relevant and $X_f = X_{f_1}\cap\ldots\cap X_{f_m}$ is affine. 
\end{proof}

\begin{remark}[Main classes of examples]\label{rmk:exs of div schemes}
Projective schemes are certainly divisorial, and so are  normal noetherian locally $\QQ$-factorial schemes with
affine diagonal (see \cite[Proposition 1.3]{BrennerSchroer}). In analogy with the case of projective schemes arising as homogeneous spectra of singly-graded rings, we will see that homogeneous spectra of ample families are divisorial schemes as well. 
\end{remark}

\begin{proposition}\label{prop:ample families and embeddings}
Let $D$ be a finitely generated abelian group and $S$ a $D$-graded ring, which is finitely generated as an $S_0$-algebra. Then $\Proj_D S$ is a divisorial scheme.
\end{proposition}
\begin{proof}
This is \cite[Corollary 3.5]{BrennerSchroer}.
\end{proof}

The following statement is a slight variation of  \cite[Theorem 4.4]{BrennerSchroer}. Since it does not follow directly from \emph{loc.~cit.}, we give a proof. The main point of this exercise is to obtain a proof of the case when $R$ is not required to be noetherian.  We rely on Proposition~\ref{prop: rtl maps into hom spectra} along with its notation.

\begin{proposition}\label{prop:embedding and ample fam}
Let $X$ be a scheme of finite type over a ring $R$ (not necessarily noetherian), $L_1,\dots,L_r$ a finite collection of invertible sheaves on $X$. If the invertible sheaves $L_1,\dots,L_r$ constitute an ample family then the following hold.
\begin{enumerate}
    \item The rational map 
    \[
    r_{\calB,\phi}\colon X \rtl \Proj_{\N^r} \bigoplus_{d\in\N^r} \Gamma(X, L^{\otimes d})
    \]
    is everywhere defined and an open embedding. Here we mean 
    \[
    \calB \equ  \bigoplus_{d\in\N^r} L^{\otimes d}\ \ \text{and}\ \ \phi \equ \id \colon  \bigoplus_{d\in\N^r} \Gamma(X, L^{\otimes d}) \to  \Gamma(X,\bigoplus_{d\in\N^r} L^{\otimes d})\ .
    \]
    \item There exists a finite index set $I$, a finite set of degrees $\{d_i\mid i\in I\}$,  a collection of global sections $\{f_i\in\Gamma(X,L^{\otimes d_i})\mid i\in I\}$, and an $\N^r$-graded polynomial algebra $A=R[T_i\mid i\in I]$ such that the rational map $\Phi\colon X\rtl \Proj_{\N^r} A$ induced by the $\N^r$-graded ring homomorphism coming from $T_i\mapsto g_i$ is everywhere defined and an embedding.  
\end{enumerate}
\end{proposition}
\begin{proof}
\textbf{Part (1):} We write 
\[
S\deq \bigoplus_{d\in\N^r} \Gamma(X,L^{\otimes d})\ . 
\]
Let $x\in X$ be an arbitrary point. We will show that $x\in U_{\calB,\phi}$, that is, that the rational map $r_{\calB,\phi}\colon X \rtl \Proj_{\N^r} S$ is defined at $x$.

By part (3) of Proposition~\ref{prop:Grauert for ample families}, there exists a $\Q$-basis $d_1,\dots,d_m$ of $\N^r$ and a collection of sections $f_j\in\Gamma(X,L^{\otimes d_j})$ for $1\leq j\leq m$ such that the subsets $X_{f_j}$ for $1\leq j\leq m$ are affine neighbourhoods of $x$. Consider the element 
\[
f \deq f_1\cdot\ldots\cdot f_m \in S\ . 
\]
Since $\deg(f)$ is in the interior of the maximal-dimensional cone spanned by the $\Q$-basis elements $d_i= \deg(f_i)$, the product $f\in S$ is a relevant homogeneous element. As 
\[
x \in X_{f} \equ X_{\phi(f)} \dsubseteq \bigcup_{f\in \Rel_{\N^r}(S)} X_{\phi(f)} \equ U_{\calB,\phi}\ ,
\]
the rational map $r_{\calB,\phi}$ is indeed defined at $x$. Therefore, $r_{\calB,\phi}\colon X\rtl \Proj_{\N^r} S$ is indeed a morphism.  

Next, pick a relevant element $f\in S$ having the property that  $X_f\subseteq X$ is an affine open subset.
The canonical morphism
\[
\phi_f\colon S_f \equ \left( \bigoplus_{d\in\N^r} \Gamma(X,L^{\otimes d}) \right)_f \simeq % 
\Gamma \big(X,\bigoplus_{d\in\N^r}L^{\otimes d}\big)_f \xlongrightarrow{\sim} %
\Gamma(X_f,\bigoplus_{d\in\N^r}L^{\otimes d}) \simeq \Gamma(X_f,\calB)
\]
is an $\N^r$-graded isomorphism, hence its restriction to degree $0$
\[
\phi_{(f)}\colon   S_{(f)} \to \Gamma(X_f,\calB)_0
\]
is an isomorphism as well. But then 
\[
\phi^{\#}\colon X_f {\simeq} \Spec \Gamma(X_f,\calB)_0 \stackrel{\sim}{\longrightarrow} \Spec S_{(f)} \equ D_+(f)
\]
is an isomorphism of schemes. Consequently, since $L_1,\dots,L_r$ is an ample family on $X$,  the collection of open subsets $$\big\{X_f\mid \text{$f\in S$ is relevant and $X_f$ is affine}\big\}$$ forms an open cover of $X$ and hence the morphism
\[
r_{\calB, \phi} \colon X \equ \bigcup_{f\in \Rel_{\N^r}(S),\, \text{$X_f$ affine}} X_f \longrightarrow
\bigcup_{f\in \Rel_{\N^r}(S),\, \text{$X_f$ affine}}  D_+(f) \subseteq \bigcup_{f\in \Rel_{\N^r}(S)} D_+(f) %
 \equ \Proj_{\N^r}S
\]
is an open immersion. 

\noindent
\textbf{Part (2):} This is again a minor modification of \cite[Theorem 4.4]{BrennerSchroer} with an eye on the fact that the base ring $R$ does not need to be noetherian. As far as the exposition goes, we borrow liberally from \cite[Section 4]{KU}. 

Assume again that $L_1,\dots,L_r$ are an ample family. Then there exists a finite set of relevant global sections $f_j$ with $1\leq j\leq m$ such that the corresponding open subsets $X_{f_j}$ form an affine cover of $X$. 
	
Consequently, we have $X_{f_j} \simeq \Spec A_i $,  or, equivalently, $\Gamma(X_{f_j},\calO_X) \simeq A_j$ as $R$-algebras, where the $A_i$'s are finitely generated algebras over $R$. Let $h_{j,1},\ldots,h_{j,m_j}$ be a generating set of $A_j$ over $R$, then for every $1\leq j\leq m$ and every $1\leq i\leq m_j$ there exist $n_{j,i}\in\N$ such that $f_j^{n_j}h_{j,i} \in S$.
	
Set 
\[
J \deq \big\{ (j,i)\in \N^2 \mid 1\leq j\leq m\, ,\, 1\leq i\leq m_j \big\}
\]  
and  
\[
g_{j,i} \deq f_j^{n_{j,i}}h_{j,i}\ \ \text{for every $(j,i)\in J$. }
\]
Consider the polynomial ring $A\deq R\big[T_{(j,i)}\mid (j,i)\in J\big]$ graded by the homomorphism
\begin{eqnarray*}
		\delta \colon  \N^{J} & \longrightarrow & \N^r \\
		e_{i,j} & \longmapsto & \deg g_{i,j}
\end{eqnarray*}  
so that the natural homomorphism $A\to S$ given by $T_{(i,j)}\mapsto g_{i,j}$ is homogeneous. 
	
Observe that although the homomorphism $A\to S$ above   might not be surjective itself,  the induced homomorphisms of rings $A_{(T_{(j,i)})} \to S_{(g_{j,i})}$  all are,  hence  the corresponding morphisms of schemes $X_{g_{j,i}}\to D_+(T_{j,i})$ are closed embeddings. As the $X_{(g_{(j,i)})}$ form an open cover of $X$, we obtain a locally closed  embedding $X\hookrightarrow \Proj_DA$. If $X$ is proper then this is a closed embedding.  \end{proof}

\begin{remark}
According to  \cite[Theorem 4.4]{BrennerSchroer},  Proposition~\ref{prop:ample families and embeddings} (1) implies the ampleness of the family $L_1,\dots,L_r$ without any further condition, and (2) implies the ampleness of $L_1,\dots,L_r$ assuming that $X$ is of finite type over a noetherian ring.  
\end{remark}

\begin{theorem}[Embedding properties of divisorial schemes, \cite{BrennerSchroer}, Corollary 4.7, \cite{HausenSchroer}] \label{thm_divisorial=>embedding}
Let $X$ be a scheme of finite type over a not necessarily noetherian ring $R$. If $X$ is divisorial, then 
\begin{enumerate}
    \item $X$ can be embedded into a simplicial toric prevariety over $R$ with affine diagonal.
    \item $X$ can be embedded into the homogeneous spectrum of a multi-graded $R$-algebra of finite type. 
\end{enumerate}
\end{theorem}
\begin{proof}
For a divisorial scheme $X$ of finite type over a not necessarily noetherian ring $R$,   Proposition~\ref{prop:ample families and embeddings} implies the existence of a locally closed embedding of $X$ into $\Proj_{\N^r} A$ for an $\N^r$-graded $R$ algebra $A$ finitely generated over $R$. This yields (2). 

Given (2), \cite[Proposition 3.4]{BrennerSchroer} implies the existence of an embedding of $X$ into a simplicial toric prevariety over $R$, the fact that the latter has affine diagonal is shown in \cite[Proposition 3.1]{BrennerSchroer}. 
\end{proof}

\begin{proposition}\label{prop:troc prev divisorial}
Let $X$ be a simplicial toric prevariety over an algebraically closed field, and let $\big\{D_\rho\mid \rho\in \Omega(S)(1)\big\}$ be the set of torus-invariant prime divisors on $X$. Then there exists positive integers $n_\rho$ such that the collection of line bundles $\calO_X(n_\rho D_\rho)$ forms an ample system on $X$. In particular a simplicial toric prevariety is divisorial. 
\end{proposition}
\begin{proof}
Given that $X$ is simplicial, for every torus-invariant prime Weil divisor $D_\rho$ there exists a positive multiple $n_\rho D_\rho$, which is Cartier, hence $\calO_X(n_\rho D_\rho)$ is a line bundle. 

Let $x\in X$ be an arbitrary point, then we can find a torus-invariant affine open subset $U_\sigma$ which contains $x$.  Since the $n_\rho D_\rho$ are effective Weil and therefore Cartier-Divisors, there exist global sections $s_\rho\in \Gamma(X,\calO_X(n_\rho D_\rho))$ such that $\supp (s_\rho)=D_\rho$. Consider the global section
\[
s \deq \prod_{\rho \notin \sigma} s_\rho\ \ \in\ \Gamma(X,\bigotimes_{\rho \notin \sigma} \calO_X(n_\rho D_\rho))\ .
\]
Then we have $X_s=U_\sigma$, which is by choice affine. Therefore, the collection of line bundles $\calO_X(n_\rho D_\rho)$ for $\rho\in\Omega(S)(1)$ forms an ample system.
\end{proof}
 
\begin{corollary}
Let $R$ be a ring (not necessarily noetherian), and let $X\deq \Proj_D R\big[T_i\mid  i\in I\big]$ be a simplicial toric prevariety over $R$. Then $X$ is a  divisorial scheme. 
\end{corollary}
\begin{proof}
Consider the simplicial toric prevariety $\Proj_D \overline{\Q}\big[T_i\mid i\in I\big]$ associated to the same $D$-grading. According to Proposition~\ref{prop:troc prev divisorial}, $\Proj_D \overline{\Q}\big[T_i\mid i\in I\big]$ is a divisorial prevariety. Since the torus-invariant prime divisors $D_\rho$, the associated line bundle $L_{\rho,n_\rho}$, and the global sections $s_\rho$ are all defined over $\Z$, the scheme $\Proj_D \Z\big[T_i\mid i\in I\big]$ (coming from the same $D$-grading) is again divisorial. But then so is its base change $X$ to $\Spec R$.  
\end{proof}
 
%%%%%%%%%%%%%%%%%%%%%%%%%%%%%%%%%%%%%%%%%%%%%%%%%%%%%%

\section{Limits of tropicalizations}

In this section, we define a notion of tropicalization for a locally closed subscheme of a toric prevariety, and study the topology of a limit of such tropicalizations. The section concludes with a proof of Theorem \ref{maintheorem} from the introduction. For the duration of this section fix $K$ to be a non-Archimedean valued field and write $R$ for its valuation ring. 

\begin{definition}
Let $Y$ be a divisorial prevariety. We define the \emph{tropicalization} $\Trop(Y,i)$ (resp. \emph{non-negative tropicalization} $\Trop^{\geq 0}(Y,i)$) of $Y$ with respect to a locally closed embedding $i\colon Y\to X$ into a toric prevariety $X$ as the image of $i^{an}(Y)$ (resp.$i^{\circ}(Y)$) under the tropicalization map $X^{an}\to X^{trop}$ (resp. non-negative tropicalization map $X^{\circ}\to X^{trop,\geq 0}$).
\end{definition}

Usually there are many ways to embed a scheme $Y$ into a toric prevariety, hence the associated tropicalization will naturally depend on the choice of the embedding $i$ (in particular on the target space $X$). To get a more intrinsic object, one can instead consider  inverse systems $\mathcal{I}$ of embeddings of $Y$ into toric prevarieties and form their associated inverse limits
\begin{equation*}
\varprojlim_{i\in\mathcal{I}}\Trop(Y,i) \qquad \textrm{ and }\qquad \varprojlim_{i\in\mathcal{I}}\Trop^{\geq 0}(Y,i)\ .
\end{equation*}

Let us elaborate on this: Suppose we are given two embeddings $i^{\prime}\colon Y\hookrightarrow X^{\prime}$ and $i\colon Y\hookrightarrow X$ into toric prevarieties. A \emph{toric morphism} of locally closed embeddings $(i^{\prime}\colon Y\to X')\rightarrow(i\colon Y\to X)$ is a toric morphism $f\colon X^{\prime}\rightarrow X$ such that $f\circ i^{\prime}=i$.

Write $\trop_i$ for the composition $\trop\circ i^{an}$ (as well as $\trop_i^{\geq 0}$ for $\trop_X^{\geq 0}\circ i^{\circ}$). Then, given a toric morphism between $i'$ and $i$, by Proposition \ref{prop_troptoricpre} (ii) and Proposition \ref{prop_nonnegtroptoricpre}, we naturally have commutative diagrams
\begin{equation*}
    \begin{tikzcd}
        & & \Trop(Y,i')\arrow[dd,"f^{trop}"]\\
        Y^{an}\arrow[rru,"\trop_{i'}"]\arrow[rrd,"\trop_{i}"']& &\\
        && \Trop(Y,i)
    \end{tikzcd}
    \qquad \textrm{ and }\qquad 
    \begin{tikzcd}
        & & \Trop^{\geq 0}(Y,i')\arrow[dd,"f^{trop}"]\\
        Y^{\circ}\arrow[rru,"\trop_{i'}^{\geq 0}"]\arrow[rrd,"\trop_{i}^{\geq 0}"']& &\\
        && \Trop^{\geq 0}(Y,i)
    \end{tikzcd}
\end{equation*}
when working over $K$ or $R$ respectively.

Thus, for a system $\mathcal{I}$ of toric embeddings (consisting of toric morphisms), the tropicalization maps $Y^{an}\to\Trop(Y,i)$ induce continuous maps
\begin{equation*}
Y^{an}\to\varprojlim_{i\in\mathcal{I}}\Trop(Y,i) \qquad \textrm{ and }\qquad Y^{\circ}\to\varprojlim_{i\in\mathcal{I}}\Trop^{\geq 0}(Y,i)
\end{equation*}
respectively, and we are interested in criteria under which these maps become homeomorphisms.

This is a natural extension to the non-separated case of the work done in \cite{FGP}, where the authors study this question for toric embeddings, i.e., closed embeddings into (separated) toric varieties and obtain sufficient conditions in terms of simple properties of the system of embeddings. We will show that, with some modifications, these conditions also work in our setup. Furthermore, we state the corresponding results for the Raynaud generic fiber and limits of non-negative tropicalizations and finally apply both results to the case of divisorial schemes.

\begin{proposition}\label{thm_inj}
Let $Y$ be a scheme of finite type over a complete non-Archimedean field $K$ or over its valuation ring $R$ and let $\mathcal{I}$ be a system of locally closed embeddings $i\colon Y\hookrightarrow X$ of $Y$ into toric prevarieties that is closed under finite products and fulfills the property:
\begin{itemize}
    \item[$(+)$] For every point $y\in Y$, there is an embedding $i\in\mathcal{I}$ and a torus-invariant open affine subset $U$ of $X$ containing $i(y)$ and such that $i\vert_{ i^{-1}(U)}:i^{-1}(U)\hookrightarrow U$ is a closed embedding.
    \end{itemize} 
Then the map $\pi_{\mathcal{I}}$ is surjective.
\end{proposition}

\begin{proof}
Let $x=(x_{i})_{i\in\mathcal{I}}$ be a point in the limit. We need to show that $\pi_{\mathcal{I}}^{-1}(x)$ is non-empty. By $(+)$, there is an embedding $i\colon Y\to X$ and a torus-invariant open affine subset $U$ of $X$ such that $x_{i}\in U^{\trop}$ and $i\vert_{i^{-1}(U)}\colon i^{-1}(U)\rightarrow U$ is a closed embedding, so that it is in particular proper. Then $\trop_i^{-1}(x_i)\subseteq (i^{an})^{-1}U^{an}$ and also the composition $\trop_{U}\circ i\vert_{i^{-1}(U)}^{an}$ is proper (by \cite{Berkovich_book}). It follows that $\trop_i^{-1}(x_i)$ is a non-empty compact subset of $X^{an}$. Now, given a finite number of embeddings $i_{1},\dots,i_{s}\in\mathcal{I}$, we have by hypothesis that $i\times i_{1}\times\dots\times i_{s}\in\mathcal{I}$ and
\begin{equation*}
\trop^{-1}(x_{i_{1}\times\dots\times i_{s}})\subseteq\trop^{-1}(x_{i})\cap \trop^{-1}(x_{i_{1}})\cap\dots\cap\trop^{-1}(x_{i_{s}})
\end{equation*}
so that this finite intersection is non-empty. Since $\trop_{i}^{-1}(x_{i})$ is compact, if follows that its intersection with all the other $\trop_{j}^{-1}(x_{j})$, which is $\pi_{\mathcal{I}}^{-1}$, is also non-empty. This shows that $\pi_{\mathcal{I}}$ is surjective. The same proof works for $Y$ over the valuation ring $R$ and  the Raynaud generic fiber $Y^{\circ}$ mapping to the non-negative troplicalizations.
\end{proof}

\begin{proposition}\label{thm_surj}
Let $Y$ be a scheme of finite type over a complete non-Archimedean field $K$ or over its valuation ring $R$ and let $\mathcal{I}$ be a system of locally closed embeddings of $Y$ into toric prevarieties that satisfies:
\begin{itemize}
    \item[$(\ast)$] There exists an open affine cover $V_{1}\cup\dots\cup V_{r}$ of $Y$ such that, for every regular function $f\in \calO_X[V_{k}]$, there exists a locally closed embedding $i\in\mathcal{I}$ such that $V_{k}$ is the preimage of an invariant open subset and $f$ is the pullback of a monomial.
\end{itemize}
Then the map $\pi_{\mathcal{I}}$ is injective.
\end{proposition}
\begin{proof}
Let $\eta$ and $\eta^{\prime}$ be distinct points in the analytification. We have to show that there exists a locally closed embedding $i\in\mathcal{I}$ such that $\trop(\eta,i)\neq\trop(\eta^{\prime},i)$. First, consider the open cover $X=U_{1}\cup\dots\cup U_{r}$ of $X$ given by condition $(\ast)$. If $\eta\in U_{k}^{an}$ and $\eta^{\prime}\not\in U_{k}^{an}$, we chose a locally closed embedding $i\in\mathcal{I}$ with the property that $U_{k}$ is the preimage of an open invariant subset $U=U_{\sigma_{1}}\cup\dots\cup U_{\sigma_{s}}$ of $Y=Y(S)$, where the $\sigma_{j}$ taken as cones of the maximal fans $\Delta_{ii}$ which do not belong to the mixed fans $\Delta_{ij}$ in the system of fans $S$. Then the point $\trop(i^{an}(\eta))$ belongs to $U_{\sigma_{k}}^{trop}$ and $\trop(i^{an}(\eta')$ does not. 

If both $\iota(\eta)$ and $i(\eta^{\prime})$ belong to a common $U_{[\sigma,i]}^{\operatorname{an}}$ and $f$ is the pullback of a monomial $a^{-1}\chi^{u}$, for some $a\neq 0$, then their images under the tropicalization map $\pi_{i}$ are the monoid homomorphisms $S_{\sigma}\to\overline{\mathbb{R}}$ which map $u$ to $\eta(f)+\val(a)$ and $\eta^{\prime}(f)+\val(a)$, respectively, and hence are distinct elements of $\Trop(X,i)$. We conclude that the map $\pi_{\mathcal{I}}$ is injective.

The proof for the Raynaud generic fiber and the non-negative tropicalizations goes through in the same way. 
\end{proof}

\begin{theorem}\label{thm_productplus&star}
Let $Y$ be a scheme of finite type over a complete non-Archimedean field $K$ or over its valuation ring $R$ and let $\mathcal{I}$ be a system of locally closed embeddings of $Y$ into toric prevarieties that is closed under finite products and satisfies $(+)$ and $(\ast)$. Then $\pi_{\mathcal{I}}$ is a homeomorphism.
\end{theorem}
\begin{proof}
By Propositions \ref{thm_inj} and \ref{thm_surj}, the continuous map $\pi_{\mathcal{I}}$ is bijective, so it remains to show that it is a homeomorphism. 

Let $Y=V_{1}\cup\cdots\cup V_{r}$ be an open affine cover of $Y$ satisfying $(\ast)$. Let $f\in\mathcal{O}_{Y}(V_{k})$ be a regular function and $i\in\mathcal{I}$ be a locally closed embedding such that $V_{k}$ is the preimage of an invariant open subset $U=U_{\sigma_{1}}\cup\cdots\cup U_{\sigma_{s}}$ of $X$ and $f$ is the pullback of a monomial $a\chi^{s}$.
Denote by $\mathcal{V}$ the preimage of $U_{\sigma_{1}}^{trop}\cup\cdots\cup U_{\sigma_{s}}^{trop}\subset X^{trop}$ under the composition
\begin{equation*}
\varprojlim_{i\in\mathcal{I}}\Trop(Y,i)\to\Trop(Y,i)\to X^{trop},
\end{equation*}
so that $\mathcal{V}$ is exactly the image of $V_k^{an}$ under the map $\pi_{\mathcal{I}}\colon Y^{an}\to\varprojlim_{i\in\mathcal{I}}\Trop(Y,i)$ (so $\calV$ is independent of the particular choice of the embedding $i$). We need to show that $V_k^{an}$ is homeomorphic to $\calV$.

The topology on $U_k^{an}$ is the coarsest that makes the evaluation maps
\begin{equation*}
    \begin{split}
        \ev_f\colon U_k^{an}&\longmapsto \R_{\geq 0}\\
        x&\longmapsto \vert f\vert
    \end{split}
\end{equation*}
for all $f\in \calO_{X}(V_k)$ continuous. Consider a locally closed embedding $i\in\mathcal{I}$ such that $f$ is the pullback of a monomial $a \chi^s$ on a torus-invariant open subset $U=U_{\sigma_{1}}\cup\cdots\cup U_{\sigma_{s}}$. Then the  evaluation map 
\begin{equation*}
\begin{split}
    U^{trop}&\longrightarrow \Rbar\\
    \Hom\big(S_{\sigma_i},\Rbar\big)\ni u&\longmapsto u(s) + \val(a)
\end{split}
\end{equation*}
is continuous and so is the composition $\calV\rightarrow U^{trop}\rightarrow \Rbar$. This composition is the pullback of $\ev_f$ along $\big(\pi_{\mathcal{I}}\vert_{V_k^{an}}\big)^{-1}$ and so this implies our claim. 

For the Raynaud generic fiber and its non-negative tropicalizations, the argument goes through in the same way.\end{proof}

\smallskip

\begin{proof}[Proof of Theorem A] Let $Y$ be a divisorial scheme of finite type over a ring $R$. It is clear that embeddings into simplicial toric prevarieties are closed under taking products and,  by Theorem \ref{thm_divisorial=>embedding} and Proposition \ref{prop:troc prev divisorial} above, they also fulfill property $(+)$. We will show that condition $(\ast)$ is fulfilled. Then Theorem \ref{maintheorem} follows from Theorem \ref{thm_productplus&star} (taking $R$ to be a non-Archimedean field for Part (i) or a valuation ring for Part (ii)).

Let $i\colon Y\hookrightarrow X$ be an embedding into a simplicial toric prevariety (coming from Theorem \ref{thm_divisorial=>embedding}), in particular, we can write $X$ as the multihomogeneous spectrum $\Proj_D S$ of a polynomial algebra $S=R[T_1,\ldots, T_n]$ graded by $D=\Z^r$ via a surjective group homomorphism $\Z^n\rightarrow \Z^r$.

Let $X=U_1\cup\cdots\cup U_r$ be the unique cover of $X$ by maximal open affine subsets and set $V_i:=i^{-1}(U_i)$. It follows from the construction (see the proof of Proposition~\ref{prop:embedding and ample fam} (2)) that $V_i$ is going to be affine as well. 

Let $f\in \calO_Y(V_{i_0})$ for a fixed $i_0$. We need to show that there is an embedding $i'\colon Y\hookrightarrow X'$ into a simplicial toric prevariety $X'$ so that $V_i$ is the preimage of a torus-invariant open subset of $X'$ and $f$ is the pullback of a monomial. 

Since $V_i\hookrightarrow U_i$ is a closed embedding of affine schemes, we will find an element $g\in \calO_X(U_i)$ whose pullback is $f$. Let $\widetilde{g}\in S$ be given by $hg$ for a monomial $h$. Consider the surjective homomorphism $\phi\colon S':=S[x]\rightarrow S$ given by $x\mapsto \widetilde{g}$. We may endow $S[x]$ with a $D$-grading such that $\deg(x)=\deg \widetilde{g}$, since $\widetilde{g}$ is homogeneous. Set $X'=\Proj_D(S')$. The induced map $Y\hookrightarrow X\rightarrow X'$ is again an embedding, since $X\rightarrow X'$ is a closed embedding. The $V_i$ are preimages of torus-invariant open subsets of $X'$ and $f\in \calO_X(V_i)$ is the pullback of the monomial $x/h$.
\end{proof}

%%%%%%%%%%%%%%%%%%%%%%%%%%%%%%%%%%%%%%%%%%%%%%%%%%%%%%

%%%%%%%%%%%%%%%%%%%%%%%%%%%%%%%%%%%%%%%%%%%%%%%%%%%%%%

%%%%%%%%%%%%%%%%%%%%%%%%%%%%%%%%%%%%%%%%%%%%%%%%%%%%%%

\bibliographystyle{amsalpha}
\bibliography{biblio}{}

\end{document}